\documentclass[12pt,a4paper]{article}
\usepackage{amsmath}
\usepackage{amsfonts}
\usepackage{amssymb}
\usepackage{amsthm}
\usepackage{mathtools}
\usepackage{mathabx}
\usepackage{mathrsfs}  %TO USE \mathscr{}
\usepackage{xcolor}
\usepackage{hyperref}
 \usepackage[bottom=4cm]{geometry}
\usepackage[latin1]{inputenc} % LaTeX, comprends les accents !
\usepackage[french,english]{babel}  
\usepackage{paralist}% Placez ici une liste de langues, la
\usepackage{natbib}
\usepackage{bbm}
\usepackage{pifont}

\definecolor{mygreen}{RGB}{28,172,0} 
\definecolor{mylilas}{RGB}{170,55,241}

\newtheorem{Assumption}{{\bf Assumption}}[section]
\newtheorem{lemma}{Lemma}[section]
\newtheorem{Remark}{{\bf Remark}}[section]
\newtheorem{definition}{{\bf Definition}}[section]
\newtheorem{theorem}{Theorem}[section]
\newtheorem{proposition}{Proposition}[section]
\newtheorem{corollary}{Corollary}[section]

\makeatletter
\renewenvironment{proof}[1][\proofname]{%
  \par\pushQED{\qed}\normalfont%
  \topsep6\p@\@plus6\p@\relax
  \trivlist\item[\hskip\labelsep\bfseries#1\@addpunct{.}]%
  \ignorespaces
}{%
  \popQED\endtrivlist\@endpefalse
}
\makeatother

\newcommand{\E}{\mathbb{E}}

\newcommand{\R}{\mathbb{R}}
\newcommand{\N}{\mathbb{N}}
\newcommand{\I}{\mathbb{I}}
\newcommand{\cf}{\mathcal{F}}
\newcommand{\ce}{\mathcal{E}}

\newcommand{\F}{\mathbb{F}}

\newcommand{\Tk}{\Theta_{\theta_k}}
\long\def\symbolfootnote[#1]#2{\begingroup\def\thefootnote{\fnsymbol{footnote}}
\footnote[#1]{#2}\endgroup}

\def\esssup{\text{ess sup}} 
\DeclareMathOperator*{\essinf}{ess\,inf}

\newcommand{\limn}{\lim_{n\rightarrow +\infty}} 
 
\newcommand{\limsupn}{\limsup_{n\rightarrow\infty}}

\author{ 
 Miryana Grigorova\footnote{Corresponding Author. Department of Statistics, University of Warwick, E-mail:miryana.grigorova@warwick.ac.uk} \\
 \and
 Ohood Aldalbahi \footnote{Department of Statistics, University of Warwick, E-mail:Ohood.Aldalbahi@warwick.ac.uk}
 \footnote{Department of Mathematics, College of Sciences, King Saud University, P.O. Box 2455, Riyadh 11451, Saudi Arabia. E-mail:oaldalbahi@ksu.edu.sa}
}

\begin{document}

\title{Non-linear optimal stopping with Bermudan strategies: the infinite horizon case}

\maketitle

\textit{Abstract:} \\
In this paper, we consider an optimal stopping problem with infinite horizon, non-negative pay-offs and non-linear evaluations $\rho_{S,\tau}$ indexed by two indices: $S$ and $\tau$, where $S$ is the time of evaluation and $\tau$ is the time when the pay-off is revealed. 
The agent's stopping strategies are constrained to be in the set 
of so-called Bermudan stopping times $\Theta$.
Under suitable assumptions on the non-linear evaluations $\rho$ and on the pay-off, we show that a dynamic programming principle holds in this framework.
We investigate the existence of $\varepsilon$-optimal stopping times, as well as the existence of optimal stopping times. 
We show that an $\varepsilon$-optimal stopping time exists. 
We also prove that the first time when the value family hits the pay-off is optimal if and only if it is finite. 
We also provide Doob's type convergence for non-negative \emph{$(\Theta, \rho)$}-supermartingales in the case where $\rho_{S,\tau}=\rho_S$ depends on the first index only. 
We provide an example from BSDEs with infinite horizon.


\section{Introduction}

Non-linear optimal stopping problems have attracted considerable interest in the recent years, both from a theoretical point view and due to their applications: cf., e.g., \cite{Grigorova-2}, \cite{Schoe}, in the discrete time case, 
and \cite{Bayraktar-2}, \cite{Bayraktar}, 
 \cite{Bayraktar-3}, \cite{belomestny2016optimal}, \cite{belomestny2017addendum}, \cite{Ekren}, \cite{Grigorova-1}, \cite{Grigorova-3}, \cite{Grigorova-4}, \cite{Nutz}, \cite{Quenez-Sulem} in the continuous time case. 
Often, these problems are considered when the terminal horizon, say $T$, is finite. 
This is particularly relevant where financial applications to pricing and hedging of American options are investigated, as the horizon $T$ corresponds in these applications to the maturity of the option (cf., e.g., \cite{EQ}, \cite{Grigorova_Am}, \cite{Rutkowski}).
Recently, \cite{Grigorova-5} consider a non-linear optimal stopping problem with finite horizon $T$, nonlinear operators $\rho_{S,\tau}$, depending on both induces (the time of evaluation $S$, and the time when the pay-off revealed $\tau$), and where the agent's strategies are constrained to lie in the so-called set of Bermudan stopping times $\Theta$ (to be defined precisely hereafter).
This can be seen as a constrained problem where the agent's strategies are in-between the discrete and the continuous stopping strategies. 
This non-linear optimal stopping problem with finite horizon was later extended to a (possibly) non-zero sum Dynkin game problem with finite horizon in \cite{Grigorova-7}
It is well-known, at least in the classical linear case, where $\rho_{S,\tau}[\cdot]$ is the usual linear conditional expectation $\E[\cdot|\cf_S]$, in the discrete-time framework, that the problem with infinite horizon is more difficult to solve than the problem with finite horizon (cf. \cite{Neveu}).  

In the present paper, we consider the non-linear optimal stopping problem with nonlinear operators $\rho_{S,\tau}$ and Bermudan strategies, and with non-negative (possibly non-integrable) pay-offs, when the horizon $T$ is equal to $+\infty$. 
We show a $(\Theta, \rho)$-Snell envelope characterization of the value family in terms of the smallest non-negative $(\Theta, \rho)$-supermartingale dominating the pay-off family, as well as a dynamic programming principle in this framework.
Our main results establish that under suitable assumptions on the pay-off $\xi$ and on the nonlinear evaluations $\rho_{S,\tau}$, there exists an $\varepsilon$-optimal stopping time.
Moreover, the first time $\nu$ when the value family hits the pay-off family is optimal if and only if $\nu$  is finite (a.s.). 
The proof is subdivided into several technical lemmas. 
It is well-known that non-negative supermartingales in the classical sense (defined with respect to the usual linear conditional expectations) converge almost surely (Doob's convergence result).
It is thus "natural" to ask oneself whether our non-negative non-linear $(\Theta, \rho)$-supermartingale also has such a property. 
We prove a Doob's type convergence for non-negative \emph{$(\Theta, \rho)$}-supermartingales in the case where $\rho_{S,\tau}=\rho_S$ depends on the first index only.
The question whether the property holds true in the case where $\rho_{S,\tau}$ depends also on the second index $\tau$ is still open.

The remainder of the paper is organized as follows:
 In Section \ref{sect_optimal}, we present the framework, formulate the optimal stopping problem, and provide some results on the value family, including the $(\theta,\rho)$-Snell envelope characterization and the dynamic programming principle (in the case of infinite horizon). 
 In Section \ref{sec_opt}, we investigate the existence of $\varepsilon$-optimal stopping times and the existence of optimal stopping times.
Section \ref{sec_Doob} is dedicated to  Doob's type convergence result for non-negative \emph{$(\Theta, \rho)$}-supermartingales.
In Section \ref{Example_1}, we provide an example induced by BSDEs with infinite horizon. 
The Appendix \ref{Appendix} includes some useful technical remarks and some of the technical proofs.

\section{The problem}\label{sect_optimal}

\subsection{The framework}\label{defin_sect}

Let $T=+\infty$ be an \textbf{infinite} terminal horizon.\\
Let $(\Omega,  \cf, P)$ be a (complete) probability space equipped with a right-continuous complete filtration $\F = \{\mathcal{F}_t \colon t\geq 0\}$. 
We set $\cf_{\infty}:= \sigma(\cup_{t \geq 0} \cf_{t})$.
\\
In the sequel, equalities and inequalities between random variables are to be understood in the $P$-almost sure sense. Equalities between measurable sets are to be understood in the $P$-almost sure sense.\\
Let $\N$ be the set of natural numbers, including $0$.  Let $\N^*$ be the set of natural numbers, excluding $0$.
Let ($\theta_k)_{k\in\N}$ be a sequence of stopping times (which are a.s. finite) 
satisfying the following properties:
\begin{itemize}
\item[(a)] The sequence $(\theta_k)_{k\in\N}$  is non-decreasing, i.e. 
for all $k\in\N$, $\theta_k\leq \theta_{k+1}$, a.s. 
\item[(b)] $\lim_{k\to\infty}\uparrow \theta_k=+\infty$ a.s. 
\end{itemize}
Moreover, we set $\theta_0=0$. \\
We note that the family of $\sigma$-algebras 
$({\cal F}_{\theta_k})_{k\in\N}$ is non- decreasing (as the sequence $(\theta_{k})$ is non-decreasing).
We denote by ${\Theta}_{\infty}:={\Theta}$ the set of stopping times $\tau$ of the form 
\begin{equation}\label{form}
\tau= \sum_{k=0}^{+\infty} \theta_k {\bf 1}_{A_k},
\end{equation}
where $A_{k} \in \mathcal{F}_{\theta_{k}}$, for each $k\in\N$, and $(A_{k})_{k \in \N}$ form a partition of $\Omega$.\\
We note that if $\tau \in \Theta$, then  $\tau$ is finite a.s. Moreover, $\tau$ is bounded from below by $0$ but is not necessarily bounded from above.

The set ${\Theta}$ can also  be described as the set of stopping times $\tau$ such that for almost all $\omega \in \Omega$, $\tau(\omega) = \theta_{k}(\omega)$, for some $k = k(\omega) \in \mathbb{N}$.

\begin{Remark} \label{Remark_set_Theta} 
Note that the set $\Theta$ is closed under concatenation:
 for each $\tau$ $\in$ $\Theta$, $\tau'$ $\in$ $\Theta$ and each 
$A \in {\cal F}_{\tau\wedge \tau'}$, the stopping time $\tau {\bf 1}_{A} + \tau' {\bf 1}_{A^c}$ is in  $\Theta$.  The set $\Theta$ is also closed under pairwise minimization (that is, for each $\tau\in\Theta$ and $\tau'\in\Theta$, we have $\tau\wedge \tau'\in\Theta$) and under pairwise maximization (that is, for each $\tau\in\Theta$ and $\tau'\in\Theta$, we have $\tau\vee \tau'\in\Theta$). Moreover,
for each non-increasing sequence of stopping times 
$(\tau_n)_{n \in {\mathbb N}} \in \Theta^{\mathbb{N}}$, we have $\lim_{n \rightarrow + \infty} \tau_n$  $\in$ $\Theta$ (cf. Remark \ref{Remark_closed_under_limit} in the Appendix for the proof).
Furthermore, if $(\tau_n)$ is a non-decreasing sequence of stopping times such that $\tau_n \uparrow \tau$ a.s. and 
$\tau$ is a.s. finite, then $\tau \in \Theta$ (cf. Remark \ref{Remark_closed_under_limit} in the Appendix for the proof).
\end{Remark}

\begin{Remark}\label{Rk_canonical} We have the following \emph{canonical}  writing of the sets in \eqref{form}: 
\begin{align*}
A_0&=\{\tau=\theta_0\};\\
%A_1&=\{\tau=\theta_1,\theta_1<T\}\backslash A_0; \;\ldots%\;  \\
A_{n+1}&=\{\tau=\theta_{n+1}\}\backslash (A_n\cup...\cup A_0); \text { for all } n\in \N.
\end{align*}  
From this writing, we have: if $\omega\in A_{k+1}$, then $\omega\notin \{\tau=\theta_k\}.$ 
\end{Remark}
\noindent
For each $\tau \in \Theta$, we denote by $\Theta_{\tau, +\infty}:=\Theta_{\tau}$ the set of stopping times $\nu \in 
\Theta$ such that $\nu \geq \tau$ a.s.\, The set $\Theta_{\tau}$ satisfies the same properties as 
the set $\Theta$.
We will call the set $\Theta$ the set of \textbf{Bermudan stopping strategies} (in the case of infinite horizon), 
and the set $\Theta_\tau$ will be referred to as 
the set of Bermudan stopping strategies from time $\tau$ perspective (in the case of infinite horizon).

We say that a family $\xi = (\xi(\tau), \, \tau \in \Theta)$ is \textbf{non-negative}, if, for each $\tau \in \Theta$, $\xi(\tau) \geq 0$ a.s. 
Note that $\xi(\tau) \geq 0$ a.s. %the non negativity property of a random variable 
means for us that $\xi(\tau)$ takes its values in $[0,+\infty]$ a.s. (in other words, our random variables take their values in the extended non-negative real line $[0,+\infty]$). \\
We denote by $L^0_{+}$ the set of all non-negative random variables.
We denote by $L^0_{+}(\cf_{\tau})$ the set of all non-negative $\cf_{\tau}$-measurable random variables. 

\begin{definition}\label{def.admi}
We say that a non-negative family 
$\phi=(\phi(\tau), \, \tau \in \Theta)$  is \emph{admissible} if it satisfies the following conditions: 
\par
1. \quad for all
$\tau \in \Theta$, $\phi(\tau)$ is a non-negative random variable (valued in $[0,+\infty]$), which is  $\mathcal{F}_\tau$-measurable. \par
 2. \quad  for all
$\tau,\tau'\in \Theta$, $\phi(\tau)=\phi(\tau')$ a.s.  on
$\{\tau=\tau'\}$.

Moreover, for $p \in [1, +\infty]$ fixed, we say that a (non-negative) admissible family $\phi$ is $p$-integrable, if for all $\tau \in \Theta$, $\phi(\tau)$ is in $L^{p}$.
\end{definition} 
Let $\phi=(\phi(\tau), \, \tau \in \Theta)$ be an admissible family. For a stopping time $\tau$  of the form \eqref{form}, we have 
\begin{equation}\label{formula}
\phi(\tau)= \sum_{k=0}^{+\infty} \phi(\theta_k ){\bf 1}_{A_k} \quad {\rm a.s.}
\end{equation}
Given two admissible families $\phi=(\phi(\tau), \, \tau \in \Theta)$ and $\phi'=(\phi'(\tau), \, \tau \in \Theta)$, we 
say that $\phi$ is {\em equal to} $\phi'$ and write $\phi = \phi'$ if, for all $\tau \in \Theta$, 
$\phi(\tau) =\phi'(\tau)$ a.s. We 
say that $\phi$ {\em dominates} $\phi'$ and write $\phi \geq \phi'$
if, for all $\tau \in \Theta$, 
$\phi(\tau) \geq \phi'(\tau)$ a.s.\\ 
The following remark can be found in \cite{Grigorova-5} and remains true in our framework as well.
\begin{Remark}\label{Rmk_on_admissibility}
Let $\phi=(\phi(\tau), \, \tau \in \Theta)$ be an admissible family. Let $\tau\in\Theta$ and let $(\tau_n)\in\Theta^{\N}$  be such that for (almost) each $\omega\in\Omega,$ there exists $n_0=n_0(\omega)$ (depending on $\omega$) satisfying, for all $n\geq n_0(\omega)$, $\tau_n(\omega)=\tau(\omega). $ Then, for all $n\geq n_0(\omega)$, $\phi(\tau_n)(\omega)=\phi(\tau)(\omega)$.\\ We show this by the following reasoning:  for each fixed $n\in\N$, let $\mathcal C_n:=\{\tau_n=\tau\}.$ For each fixed $m\in\N$, let
$\mathcal A_m:=\cap_{n\geq m}\mathcal C_n=\cap_{n\geq m}\{\tau_n=\tau\}.$ Note that  the set  $\mathcal A_m$ might be empty. We have $\cup_{m\in\N}\mathcal A_m=\Omega.$ Moreover, by the admissibility of $\phi$, we have, for each fixed $n\in\N$, $\phi(\tau_n)=\phi(\tau)$, on $\mathcal C_n=\{\tau_n=\tau \}.$ Hence, for each fixed  $m\in\N$,
 \begin{equation}\label{Eq_small_property}
  \text{ for all }n\geq m, \phi(\tau_n)=\phi(\tau)  \text{  on } \mathcal A_m=\cap_{n\geq m}\mathcal C_n.
 \end{equation} 
  Let $\omega\in\Omega$. By assumption, there exists  $n_0=n_0(\omega)$ such  that $\omega\in\mathcal A_{n_0}.$ By property \eqref{Eq_small_property} (applied with $m=n_0$), for all $n\geq n_0$, $\phi(\tau_n)(\omega)=\phi(\tau)(\omega),$ which is the desired conclusion.  
\end{Remark}

\subsection{The optimisation problem} \label{optimisation_problem}
Let  $\xi = (\xi(\tau), \, \tau \in \Theta)$ be a \textbf{non-negative admissible family} modelling an agent's dynamic financial position.
\begin{Remark} 
For example, the family  $\xi$ can be defined 
via a given progressive process $(\xi_t)_{t\in[0,T]}$ as follows: 
for each $\tau$ $\in$ $\Theta$, we set  $\xi(\tau) := \xi_\tau$, which models the gain of the agent if they stop at $\tau$.
The family of random variables $\xi=(\xi(\tau), \, \tau \in \Theta)$ can be shown to be admissible. 
If for each $k \in \mathbb{N}$, $\xi_{\theta_{k}}$ is non-negative, then the family of random variables $\xi = (\xi(\tau), \, \tau \in \Theta)$ is non-negative.
If for each $k \in \mathbb{N}$, $\xi_{\theta_{k}} \in L^{p}$, then the admissible family $\xi$ is $p$-integrable. The financial interpretation of this example is as follows: 
the agent can choose their strategy only among the stopping times in $\Theta$,  that is, among the stopping times which, for almost each $\omega$, have values in the 
\textbf{infinite countable} grid $\{0, \theta_1(\omega), \ldots, \theta_{n}(\omega), \theta_{n+1}(\omega),\ldots\}$.  In this example, the financial position which is actually taken into account in the problem
corresponds to the values of the process $(\xi_t)$ only at times $0, \theta_1,..., \theta_{n}, \theta_{n+1},...$.
\end{Remark}
\noindent
The \textit{minimal risk at time} $0$ over all Bermudan stopping strategies  is defined by:
\begin{equation}\label{optimal_stopping_problem_0bis}
\tilde V(0):= \inf_{\tau\in \Theta} \tilde \rho_{0,\tau}(\xi(\tau))=  -V(0),
\end{equation}
where
\begin{equation}\label{optimal_stopping_problembis}
V(0):=\sup_{\tau\in\Theta} \rho_{0,\tau}[\xi(\tau)],
\end{equation}
and where $\rho_{0,\tau}[\cdot]=-\tilde\rho_{0,\tau}[\cdot].$\\
[0.2cm]
We can view these problems also as \textit{constrained} optimal stopping problems, where the agent's strategy is constrained to remain in the set $\Theta$. This terminology has been used, for instance, in \cite{menaldi2016Robin}.
\\
We introduce the following properties  on the non-linear operators $\rho_{S, \tau}[\cdot]$, which will appear in the sequel. \\
[0.2cm] 
For
$S\in\Theta$, $S'\in\Theta$, $\tau\in\Theta$, for $\eta$, $\eta_{1}$ and $\eta_2$ in $L^0_{+}(\cf_\tau)$, for $\xi=(\xi(\tau))$ a non-negative admissible family: %\in\mathcal{S}^2$:  %$(\tau_n)\in\stopo^N$, and $\eta_n\in L^2(\cf_{\tau_n}), $ 
\begin{compactenum}[(i)]
\item[(i)] $\rho_{S,\tau}: L^{0}_{+}(\cf_\tau)\longrightarrow L^{0}_{+}(\cf_S)$ 
\item[(ii)] \emph{(admissibility)} $\rho_{S,\tau}[\eta]=\rho_{S',\tau}[\eta]$ a.s.  on $\{S=S'\}$.  
\item[(iii)] \emph{(knowledge preservation)}
$\rho_{\tau,S}[\eta]=\eta, $
for all $\eta\in L^0_{+}(\cf_S)$, all $\tau\in\Theta_S.$
%for all $\eta\in L^2(\cf_S)$, all $\tau\in\Theta_S.$ 
\item[(iv)] \emph{(monotonicity)}  $\;\rho_{S,\tau}[\eta_1]\leq \rho_{S,\tau}[\eta_2]$ a.s., if  $\eta_1\leq \eta_2$ a.s. 
\item[(v)] \emph{(consistency)}  $\;\rho_{S,\theta}[\rho_{\theta,\tau}[\eta]]= \rho_{S,\tau}[\eta]$, for all $S, \theta, \tau$ in $\Theta$ such that $S\leq \theta\leq \tau$ a.s. %(\textit{"stronger":  $\rho_{S,\theta}[\rho_{\theta,\tau}(\eta)]\I_{\{S\leq \theta\leq \tau\}}= \rho_{S,\tau}[\eta]\I_{\{S\leq \theta\leq \tau\}}$; if not the property "stronger", we assume that $\rho(0)=0$)).}  
\item [(vi)] \emph{("generalized" zero-one law) }$\; I_A\rho_{S,\tau}[\xi(\tau)] =I_A\rho_{S,\tau'}[\xi(\tau')],$ for all $A\in\cf_S$, $\tau\in \Theta_{S}$, $\tau'\in\Theta_{S}$ such that $\tau=\tau'$ on $A$.

\item[(vii)] \emph{(monotone Fatou property with respect to the terminal condition)}\\
 $\rho_{S, \tau}[\eta] \leq \liminf_{n \to +\infty} \rho_{S, \tau}[\eta_{n}]$, for $(\eta_{n}), \eta$ such that $(\eta_{n})$ is non-decreasing, $\eta_{n} \in L^{0}_{+}(\cf_{\tau})$, and $\lim_{n \to +\infty} \uparrow \eta_{n} = \eta$ a.s.\\
\end{compactenum}
\noindent
Fatou property is often assumed in the literature on risk measures (particularly in the case where the random variables are in $L^{\infty}$).

\begin{Remark}\label{Remark_CB}
Note also that if $\rho$ satisfies \textbf{monotonocity} (iv) and \textbf{monotone Fatou} property with respect to  the terminal condition (vii), then $\rho_{S, \tau}[\eta]= \lim_{n \to +\infty} \rho_{S, \tau}[\eta_{n}]$, for $(\eta_{n}), \eta$ such that $(\eta_{n})$ is non-decreasing, $\eta_{n} \in L^{0}_{+}(\cf_{\tau})$, and $\lim_{n \to +\infty} \uparrow \eta_{n} = \eta$ a.s. Indeed, by monotonicity of $\rho_{S, \tau}[\cdot]$, we have $\rho_{S, \tau}[\eta_n]\leq \rho_{S, \tau}[\eta]$; hence, $\limsup_{n\to+\infty}\rho_{S, \tau}[\eta_n]\leq \rho_{S, \tau}[\eta]. $ On the other hand, by   
(vii), $\rho_{S, \tau}[\eta]\leq \liminf_{n\to+\infty}\rho_{S, \tau}[\eta_n]. $ Hence, $\rho_{S, \tau}[\eta]= \lim_{n\to+\infty}\rho_{S, \tau}[\eta_n]$. Such type of property is also known in the literature (e.g. risk measures) as \textbf{continuity from below}. 
\end{Remark}

Let us emphasize that no assumptions of convexity (or concavity) or translation invariance of the non-linear operators $\rho$ are made.

\subsection{ Preliminary study of the problem and its value}\label{subsect2}

As is usual in optimal control, we embed the above optimization problem \eqref{optimal_stopping_problembis} in a larger class of problems by 
considering for each $\nu \in \Theta$, the random variable $V(\nu)$, where 
\begin{equation}\label{optimal_stopping_problem_kbis}
V(\nu):=\esssup_{\tau\in \Theta_{\nu}} \rho_{\nu,\tau}[\xi(\tau)].
\end{equation} 
We note that, if $\rho$ satisfies the property of knowledge preservation (property (iii)), then $\rho_{\nu,\nu}[\xi(\nu)]=\xi(\nu)$.
\begin{lemma}\label{Lemma_admissible}(Admissibility of $V$)
Under the assumption of admissibility (ii) and "generalized" zero-one law (vi) on the non-linear operators, the family of random variables
$V:=  (V(\nu), \, \nu \in \Theta)$ defined in \eqref{optimal_stopping_problem_kbis} is \emph{admissible} in the sense of Definition \ref{def.admi}.

\end{lemma}
\noindent
The proof uses arguments similar to those of  Lemma 8.1 in \cite{Grigorova-3}, combined with some properties of the non-linear operators $\rho$, with suitable modifications to account for the infinite horizon case.

\begin{proof} 
Property 1. of the definition of admissibility follows from the definition of the essential supremum, the random variables of the family   $(\rho_{\nu,\tau}[\xi(\tau)], \tau\in\Theta_\nu)$ being $\cf_\nu$-measurable. \\
Let us prove Property 2. Let $\nu$ and $\nu'$ be two stopping times in $\Theta$. We set $A:=\{\nu=\nu'\}$ and we show that $V(\nu)=V(\nu')$, a.s. on $A$. 
We have   
\begin{equation}\label{eq_toto1}
\begin{aligned}
\I_AV(\nu)&=\I_A\esssup_{\tau\in \Theta_{\nu}} \rho_{\nu,\tau}[\xi(\tau)] =\esssup_{\tau\in \Theta_{\nu}} \I_A\rho_{\nu,\tau}[\xi(\tau)]=\esssup_{\tau\in \Theta_{\nu}} \I_A\rho_{\nu',\tau}[\xi(\tau)],
\end{aligned}
\end{equation}
where we have used the admissibility property of $\rho$ for the last equality.
\\
Let $\tau\in\Theta_\nu$. We set $\tau_A:= \tau\I_A+\nu'\I_{A^c}$. 
We note that $\tau_A\in\Theta$ (by the property of concatenation of $\Theta$); 
moreover,  $\tau_A\in\Theta_{\nu'}$ and $\tau_A=\tau$ a.s. on $A$ (by the definition of  $\tau_A$). %We note that  $\tau_A \geq \nu'$ a.s.
\\
Using this, the admissibility of the family $\xi$, and the "generalized" zero-one law property of $\rho$, we get
$
\I_A\rho_{\nu',\tau}[\xi(\tau)]=\I_A\rho_{\nu',\tau_A}[\xi(\tau_A)]\leq \I_AV(\nu').
$
As $\tau\in\Theta_\nu$ is arbitrary, we conclude that 
$\esssup_{\tau\in \Theta_{\nu}}\I_A\rho_{\nu',\tau}[\xi(\tau)]\leq \I_AV(\nu').$
Combining this inequality with \eqref{eq_toto1} gives 
$\I_AV(\nu)\leq \I_AV(\nu').$ We obtain the converse inequality by interchanging the roles of $\nu$ and $\nu'$. 
\end{proof}
%%%%
\noindent
Under the assumptions of the above lemma, the following remark holds true.

\begin{Remark}\label{Rmk_consequence_admissible}
As a consequence of the \textbf{admissibility} of the value family $V$, we have: 
for each  $k\in\N, $ it holds
$V(\nu)= V (\theta_k)$ a.s. on $\{\nu=\theta_k\}$.  Hence, under the assumptions of Lemma \ref{Lemma_admissible}, for $\nu\in\Theta$ of the form $\nu=  \sum_{k=0}^{+\infty} \theta_k {\bf 1}_{A_k}$, we have $V(\nu)=\sum_{k=0}^{+\infty} V(\theta_k ){\bf 1}_{A_k}$.
\end{Remark}

\begin{Remark}\label{TEXRemark33}
1. Under the assumption of knowledge preservation (iii) on $\rho$, we have $V(\theta_k) \geq \xi(\theta_k)$, for each $k\in\N.$\\
Indeed, $V(\theta_{k}) = \esssup_{\tau \in \Theta_{\theta_{k}}}\rho_{\theta_{k}, \tau}[\xi(\tau)] \geq \rho_{\theta_{k}, \theta_{k}}[\xi(\theta_{k})]$, and by the property (iii) of the non-linear operators, we have $\rho_{\theta_{k}, \theta_{k}}[\xi(\theta_{k})] = \xi(\theta_{k})$. Hence, $V(\theta_{k}) \geq \xi(\theta_{k})$.\\
[0.2cm]
2.  If, moreover, $\rho$ satisfies the properties of admissibility (ii) and "generalized" zero-one law (vi), then, for each $\tau\in\Theta,$  $V(\tau) \geq \xi(\tau)$. \\
This follows from the first statement of this remark, and from the admissibility of $\xi$ and of $V$ (cf. Lemma \ref{Lemma_admissible} and Remark \ref{Rmk_consequence_admissible}).  
Hence, under these assumptions, if the family $\xi$ is non-negative, then the value family $V$ is non-negative. 
\end{Remark}

\noindent

\noindent
\noindent
Let us recall the notion of \emph{$(\Theta, \rho)$-(super)martingale family} from \cite{Grigorova-5}, which formally remains the same in our infinite horizon framework and with random variables valued in $[0,+\infty]$. 
\begin{definition}\label{def_supermartingaleBIS}
Let $\phi=(\phi(\tau), \, \tau \in \Theta)$ be a \emph{non-negative admissible} family.
 We say that $\phi$
 is a \emph{$(\Theta, \rho)$}-supermartingale (resp. \emph{$(\Theta, \rho)$}-martingale) \emph{family} if 
 for all $\sigma, \tau$ in $\Theta$ such that $\sigma\leq\tau$ a.s., we have 
$$\rho_{\sigma,\tau}[\phi({\tau})]\leq \phi({\sigma}) \text{ (resp. }=\phi(\sigma)) \text{ a.s.} $$
\end{definition}
\noindent

The following properties on the value family $V$ have been established in \cite{Grigorova-5} for the finite case (under appropriate integrability assumptions) and hold true also in our infinite horizon framework with pay-offs valued in $[0,+\infty]$. 

\begin{theorem}\label{Theorem_supermartingale_Snellenvelope}
  \begin{enumerate}
      \item [1.]\emph{($(\Theta, \rho)$-supermartingale)}
     Under the assumption of admissibility (ii),
     consistency (v),
      "generalized" zero-one law (vi),
     and monotone Fatou property with respect to the terminal condition (vii)
      on the non-linear operators,
      the value family $V$ is a \emph{$(\Theta, \rho)$}-supermartingale family.
      \item [2.]\emph{($(\Theta, \rho)$-Snell envelope)}
       If, moreover, the non-linear operators also satisfy the properties of 
        knowledge preservation (iii),
        and monotonicity (iv),
          the value family $V$ is equal to the \emph{$(\Theta, \rho)$-Snell envelope} of the family $\xi$, that is, 
          the smallest \emph{$(\Theta, \rho)$}-supermartingale family dominating the family $\xi=(\xi(\tau), \tau \in \Theta)$.
  \end{enumerate}
\end{theorem}
To prove this theorem, we first state a useful lemma.
\begin{lemma}{(Maximizing sequence lemma)}\label{Maximizing_sequence_lemma}
    Under the assumption of "generalized" zero-one law (vi) of the non-linear operators,
    there exists a maximizing sequence for the value $V(S):=\esssup_{\tau\in \Theta_{S}} \rho_{S,\tau}[\xi(\tau)]$.
\end{lemma}
The proof of this lemma is similar to that of Lemma 2.3 in \cite{Grigorova-4}, Lemma 2.11 in \cite{Grigorova-5} and is given for the convenience of the reader. It is enough to show that the family $(\rho_{S, \tau}[\xi(\tau)])_{\tau \in \Theta_{S}}$ is stable under pairwise maximization.
     Then, the result follows from a well-known property of the essential supremum.
\begin{proof} 
    Let $\tau$ and $\tau'$ be in $\Theta_{S}$. 
    Set $A:=\{\rho_{S, \tau'}[\xi(\tau')] \leq \rho_{S, \tau}[\xi(\tau)]\}$
    and
    let $\nu:= \I_A \tau + \I_{A^c} \tau'$.
    Clearly, $A \in {\cal F}_{S}$.
    Furthermore, $\nu \in \Theta_{S}$ (by the concatenation property of the set $\Theta_{S}$).
    Also, $\nu= \tau$ on $A$, and $\nu=\tau'$ on $A^c$.
 By the "generalized" zero-one law of the non-linear operators, we get:
\begin{equation}
 \begin{split}
\rho_{S, \nu}[\xi(\nu)]
 = \I_A \rho_{S, \nu}[\xi(\nu)] + \I_{A^c} \rho_{S, \nu}[\xi(\nu)] 
 &= \I_A \rho_{S, \tau}[\xi(\tau)] + \I_{A^c} \rho_{S, \tau'}[\xi(\tau')] 
 \\&= \max\left(\rho_{S, \tau'}[\xi(\tau')], \rho_{S, \tau}[\xi(\tau)]\right).
\end{split}
\end{equation}
This shows the stability under pairwise maximization of the value family \\ $(\rho_{S, \tau}[\xi(\tau)])_{\tau \in \Theta_{S}}$.
\end{proof}
\begin{proof}[Proof of Theorem \ref{Theorem_supermartingale_Snellenvelope}]
The proof of the theorem relies on the maximizing sequence lemma and on properties of $\rho$, by following the same arguments as those of the proof of Theorem 2.10 in \cite{Grigorova-5} and is omitted. 
\end{proof}

\subsection{The Dynamic Programming Principle in the case of infinite horizon} 
\begin{definition}[Dynamic Programming Principle in the case of infinite horizon]
We say that an admissible non-negative family $\phi$ satisfies the Dynamic Programming Principle (abridged DPP) in the case of infinite horizon, if the following property holds true: \\
For all $k \in \N$, 
\begin{equation}
    \phi(\theta_k)=\max\left(\xi(\theta_k), \rho_{\theta_k,\theta_{k+1}}[\phi(\theta_{k+1})]\right).
\end{equation}
\end{definition}

In this sub-section, we will show that, in the infinite horizon case, the value family satisfies the DPP. 
 In order to do this, we are first interested in "what happens on the right of $V(\theta_k)$", for each $k \in \N$.
\\
Let $k \in \N$ be fixed. We define
\[
\Theta_{\theta_k^+} :=\{\tau \in \Theta_{\theta_k}: \tau>\theta_k \}.
\]
Let us note that this set is non-empty. 
We define the strict value $V^+(\theta_k)$ at $\theta_k$ by:
\[
V^+(\theta_k) := \esssup_{\tau\in \Theta_{\theta_k^+}} \rho_{\theta_k,\tau}[\xi(\tau)].
\]

\begin{Remark}\label{Rmk_sets}
    We have $\Theta_{\theta_k^+}=\Theta_{\theta_{k+1}}$.
   \\ Indeed, let $\tau \in \Theta_{\theta_k^+}$. 
    Then, $\tau$ can be written as: $\tau =  \sum_{i=k+1}^{+\infty} \theta_i {\bf 1}_{A_i}$, 
    where for each $i \in \N$ such that $i\geq k+1$, $A_i \in \cf_{\theta_i}$, and $(A_{i})_{i \geq k +1}$ form a partition of $\Omega$.
    Hence, $\tau \in \Theta_{\theta_{k+1}}$.  
    Conversely, if $\tau \in \Theta_{\theta_{k+1}}$, then clearly $\tau \in \Theta_{\theta_k^+}$.
    \\
Due to this remark, we get 
\begin{equation} \label{Eq_V+_defn}
    V^+(\theta_k) 
    = \esssup_{\tau\in \Theta_{\theta_k^+}} \rho_{\theta_k,\tau}[\xi(\tau)]
    = \esssup_{\tau\in \Theta_{\theta_{k+1}}} \rho_{\theta_k,\tau}[\xi(\tau)].
\end{equation}
\end{Remark}

\begin{lemma} \label{Maximizing_sequence_lemma_for_V+}
    Under the assumption of "generalized" zero-one law (vi) on the non-linear operators,
    there exists a maximizing sequence for $V^+(\theta_k)$.
\end{lemma}
\begin{proof}
  The proof of this lemma is similar to the proof of the existence of a maximizing sequence for $V(\theta_k)$, and is left to the readers. (We also refer to the proof of Lemma 2.3 in \cite{Grigorova-4} for similar arguments).
\end{proof}

The following proposition establishes that the strict value $V^+(\theta_k)$ at $\theta_k$ is equal to the non-linear evaluation from $\theta_k$ perspective of the value $V(\theta_{k+1})$.

\begin{proposition} \label{Prop_1}
    Under the assumptions of monotonicity (iv),
     consistency (v),
      "generalized" zero-one law (vi),
     and monotone Fatou property with respect to the terminal condition (vii),
      on the non-linear operators,
      we have
       \[
     V^+(\theta_k)= \rho_{\theta_k,\theta_{k+1}}[V(\theta_{k+1})].
      \]
\end{proposition}
\begin{proof}
The proof of the proposition relies on the maximizing sequence lemmas for $V(\theta_{k+1})$ and $V^+(\theta_k)$ (cf. Lemma \ref{Maximizing_sequence_lemma} and Lemma \ref{Maximizing_sequence_lemma_for_V+}), Eq. \eqref{Eq_V+_defn} (cf. Remark \ref{Rmk_sets})
and on properties of $\rho$, by following the same arguments as those of the proof of Proposition 2.15 in \cite{Grigorova-5} and is omitted. 
\end{proof}
\begin{proposition} \label{Prop_2}
     Under the assumption of  knowledge preservation (iii) and
      "generalized" zero-one law (vi)
      on the non-linear operators, 
      we have
      \[
    V(\theta_k)= \xi(\theta_k) \vee V^+(\theta_k).
      \]
\end{proposition}
The proof of this proposition is similar to the proof of Proposition 2.16 in \cite{Grigorova-5} (in the finite horizon case). 
\begin{proof} 
Since $\rho$ satisfies the knowledge preservation property (iii), we can apply Remark \ref{TEXRemark33}. 
By Remark \ref{TEXRemark33}, first statement, we have $V(\theta_k) \geq \xi(\theta_k)$.    
On the other hand, since $\Theta_{\theta_{k+1}} \subset \Theta_{\theta_k}$, we have $V(\theta_k) \geq V^+(\theta_k)$. 
By combining these inequalities, we get $V(\theta_k) \geq \xi(\theta_k) \vee V^+(\theta_k)$. 
\\
It remains to show the converse inequality. 
    Let $\tau \in \Theta_{\theta_k}$ and $\tau' \in \Theta_{\theta_{k+1}}$. 
    We define $\nu:=  \tau \I_{\{\tau > \theta_k\}} +  \tau' \I_{\{\tau \leq \theta_k\}}$. 
    We have $\nu \in \Theta_{\theta_k^+}$.  
   By the concatenation property of $\Theta$, we have  $\nu \in \Theta$.
   Moreover, on $\{\tau > \theta_k\}=\{\tau \geq \theta_{k+1}\}$, $\nu = \tau \geq \theta_{k+1}$ 
   and on $\{\tau \leq  \theta_k\}$, $\nu =\tau' \geq \theta_{k+1}$ (as $\tau' \in  \Theta_{\theta_{k+1}} $).
   Hence, $\nu \in \Theta_{\theta_{k+1}}$.
  Therefore, we have
    \[
      V^+(\theta_k) 
    = \esssup_{\tau \in \Theta_{\theta_k^+}} \rho_{\theta_k,\tau}[\xi(\tau)]
     = \esssup_{\tau \in \Theta_{\theta_{k+1}}} \rho_{\theta_k,\tau}[\xi(\tau)]
    \geq \rho_{\theta_k,\nu}[\xi(\nu)].
    \]
    Hence, we have 
    \begin{equation} \label{eq_1}
       \I_{\{\tau > \theta_k\}}  V^+(\theta_k)  \geq  \I_{\{\tau > \theta_k\}} \rho_{\theta_k,\nu}[\xi(\nu)].
    \end{equation}
    Furthermore, on the set $\{\tau > \theta_k\}$, we have $\nu = \tau$, so the "generalized" zero-one law yields
    \begin{equation}\label{eq_2}
        \I_{\{\tau > \theta_k\}} \rho_{\theta_k,\nu}[\xi(\nu)]
        = \I_{\{\tau > \theta_k\}} \rho_{\theta_k,\tau}[\xi(\tau)].
    \end{equation}
    By combining \eqref{eq_1} and \eqref{eq_2}, we get
    \begin{equation}\label{eq_3}
        \I_{\{\tau > \theta_k\}} \rho_{\theta_k,\tau}[\xi(\tau)]
        \leq \I_{\{\tau > \theta_k\}}  V^+(\theta_k).  
    \end{equation}
    On the other hand, since $\tau \in \Theta_{\theta_k}$, we have 
    \[
    \rho_{\theta_k,\tau}[\xi(\tau)] 
    = \I_{\{\tau = \theta_k\}} \rho_{\theta_k,\tau}[\xi(\tau)] + \I_{\{\tau > \theta_k\}}  \rho_{\theta_k,\tau}[\xi(\tau)].
    \]
 By using the "generalized" zero-one law and the knowledge preservation of the non-linear operator $\rho_{\theta_k,\tau}[\cdot]$, we get
\begin{equation}\label{eq_4}
\I_{\{\tau = \theta_k\}} \rho_{\theta_k,\tau}[\xi(\tau)]
 = \I_{\{\tau = \theta_k\}} \rho_{\theta_k,\theta_k}[\xi(\theta_k)]
 = \I_{\{\tau = \theta_k\}} \xi(\theta_k).
\end{equation}
From \eqref{eq_3} and \eqref{eq_4}, we obtain 
\[ 
 \begin{split}
 \rho_{\theta_k,\tau}[\xi(\tau)] 
    &= \I_{\{\tau = \theta_k\}} \rho_{\theta_k,\tau}[\xi(\tau)] + \I_{\{\tau > \theta_k\}}  \rho_{\theta_k,\tau}[\xi(\tau)]
    \\ & \leq  \I_{\{\tau = \theta_k\}} \xi(\theta_k) +   \I_{\{\tau > \theta_k\}}  V^+(\theta_k)
     \\ & \leq \xi(\theta_k) \vee  V^+(\theta_k).
\end{split}
\]
By taking the essential supremum over $\tau \in \Theta_{\theta_k}$, we obtain  $V(\theta_k) \leq \xi(\theta_k) \vee V^+(\theta_k)$. 
Thus, the proof is completed. 
\end{proof}

We obtain the following theorem by combining Proposition \ref{Prop_1} and Proposition \ref{Prop_2}.

\begin{theorem}[DPP in the case of infinite horizon]\label{Theorem_DPP_infinite_horizon}
    Under the assumptions of  knowledge preservation (iii),
    monotonicity (iv),
     consistency (v),
      "generalized" zero-one law (vi),
     and monotone Fatou property with respect to the terminal condition (vii),
      on the non-linear operators,
      the value family $V$ satisfies the DPP in the case of infinite horizon:
      \[
     V(\theta_k)= \xi(\theta_k) \vee \rho_{\theta_k,\theta_{k+1}}[V(\theta_{k+1})], \quad \text{for each} \quad k \in \N.
      \]
\end{theorem}

\section{The $\varepsilon$-optimal stopping times and optimal stopping times}\label{sec_opt}
In this section, we investigate the existence of $\varepsilon$-optimal stopping times and optimal stopping times, 
under suitable assumptions, for the optimization problem (at time $\nu=\theta_k$): 
\begin{equation}\label{optimal_stopping_problem_k}
V(\theta_k):=\esssup_{\tau\in \Theta_{\theta_k}} \rho_{\theta_k,\tau}[\xi(\tau)].
\end{equation} 

We introduce the following assumption on the non-linear operator $\rho$. 
\begin{Assumption}\label{Assumption_rho_C}
 For any positive number $b$, for $\eta \in L^{+}_0(\cf_{\tau})$,
 \[
\rho_{S, \tau}[\eta + b]
\leq \rho_{S, \tau}[\eta]+Cb,
    \]
    where $C$ is a positive constant, depending on $\rho$ only.
\end{Assumption}
\begin{Remark}
    For any sub-additive operator $\rho_{S, \tau}[\cdot]$ this assumption is true, with $C=1$. 
    In the case where $\rho_{S, \tau}[\cdot]=\rho_{S}[\cdot]=\E[\cdot|\cf_S]$, the inequality in Assumption \ref{Assumption_rho_C} is trivially an equality (with $C=1)$. 
   \\ In the case where $\rho_{S, \tau}[\cdot]$ is the $g$-conditional expectation on infinite horizon (cf. Section \ref{Example_1} Eq. \eqref{eq_g_conditional_expectation}), we have:  $p=2$, $\eta \in L^2$ and  $\rho_{S, \tau}[\cdot]=\rho_{S}[\cdot]=\ce_S^g[\cdot]$, and
   the property from Assumption \ref{Assumption_rho_C} follows by Remark \ref{remark_C} (cf. Section \ref{Example_1}). 
\end{Remark}

\begin{theorem}\label{Theorem_optimal_stopping_times}
\begin{enumerate}
    \item Assume that $\limsupn \xi(\theta_n) <+\infty$.
    We assume that: if $\eta \in L^+_0(\cf_{\infty})$, if 
  $(\tau_k)$ $\in \Theta^{\mathbb{N}}$ non-decreasing such that $\lim_{k \rightarrow +\infty} \tau_k  = +\infty$, then
  \[
  \limsup_{n \rightarrow +\infty} \limsup_{k \rightarrow +\infty} \rho_{\theta_n, \tau_k} [\eta] =  \eta.\]
    Let $\varepsilon >0 $. For each $k \in \N$, we define: 
    \begin{equation}\label{Defn.epsilon_stopping_time}
    A_k^{\varepsilon}:=\{\tau \in \Tk : V(\tau) < \xi(\tau)+\varepsilon \,\,\, a.s. \}
   \,\,\,\, \text{and} \,\,\,\,
    \nu_k^{\varepsilon}:= \essinf A_k^{\varepsilon}.
   \end{equation}
    Then, $\nu_k^{\varepsilon}$ is a.s. finite and 
    \[
 V(\theta_k) \leq  \rho_{\theta_k, \nu_k^{\varepsilon}} [\xi(\nu_k^{\varepsilon})]+C\varepsilon,
    \]
    where $C>0$ is a positive constant (depending on $\rho$ only).
    \item
    Let us define: 
    \begin{equation}\label{Defn.stopping_time}
    A_k:=\{\tau \in \Tk : V(\tau) = \xi(\tau) \,\,\, a.s. \}
   \,\,\,\, \text{and} \,\,\,\,
    \nu_k:= \essinf A_k,
   \end{equation}
    with the convention that $\essinf A_k = + \infty$ if $A_k=\emptyset$.
   It holds: \\ $\nu_k$ is optimal for the problem with value $V(\theta_k)$, if and only if, $\nu_k < +\infty$ a.s.
\end{enumerate}
\end{theorem}

The proof of Theorem \ref{Theorem_optimal_stopping_times} passes through several intermediary results. 
\begin{Remark}\label{Remark_RC}
    Any admissible family $(\phi(\tau),\tau \in \Theta)$ in our infinite horizon framework is \textbf{right-continuous along Bermudan stopping times}, that is, for all $\tau \in \Theta$, and for all non-increasing sequences of Bermudan stopping times $(\tau_n)$ $\in$ $\Theta$ such that $\tau_n \downarrow \tau$, it holds $\limn \phi(\tau_n)=\phi(\tau)$. 
\\ Indeed, let $\tau \in \Theta$, and let $(\tau_n) \in \Theta^{\N}$ be such that $\tau_n \downarrow \tau$. For each $n$, we have $\tau_n = \sum_{l=0}^{+\infty} \theta_l \I_{A_l^{(n)}} $ and $\tau = \sum_{l=0}^{+\infty} \theta_l \I_{A_l}$ (cf. the canonical writing from Remark \ref{Rk_canonical}).
\\ Let $\omega \in \Omega$, there exists a unique $l_0=l_0(\omega)$ such that %$\omega \in A_{l_0}$ , and 
$\tau(\omega)=\theta_{l_0}(\omega)$. Then, as $\tau_n(\omega) \downarrow \tau (\omega)$ and as $\theta_k(\omega) \uparrow +\infty$, after a certain rank $n_0=n_0(\omega)$, $\tau_n(\omega) = \tau (\omega)=\theta_{l_0} (\omega)$. 
Hence, there exists $n_0=n_0 (\omega)$ such that for all $n \geq n_0$, $\tau_n(\omega) = \tau (\omega)$, and, 
hence, by Remark \ref{Rmk_on_admissibility}, for all $n \geq n_0$, $\phi(\tau_n)(\omega) = \phi(\tau) (\omega)$. We conclude that  $\limn \phi(\tau_n)(\omega)=\phi(\tau)(\omega)$.
\end{Remark}

\begin{Remark}\label{Remark_v_k_finite}
    If $\nu_k < +\infty$ a.s., then $V(\nu_k)=\xi(\nu_k)$.
\end{Remark}

\begin{proposition}[Optimality criterion]\label{prop_optimality_criterion}
     Let $\nu^*$ be a \textbf{finite} stopping time. The finite stopping time $\nu^*$ is optimal for the problem \eqref{optimal_stopping_problem_k} with value $V(\theta_k)$ if the following three conditions hold:
      \begin{itemize}
      \item[i)] $\nu^* \in \Tk$.
          \item [ii)] $V(\nu^*)=\xi(\nu^*)$ a.s.
          \item[iii)] The value family $V$ is a \emph{$(\Theta, \rho)$}-martingale on $[\theta_k, \nu^*]$. 
      \end{itemize}
\end{proposition}
\begin{proof}
    The proof follows the same arguments as the proof of Lemma 2.23, Statement 1, in \cite{Grigorova-5}, and is omitted.
\end{proof}
\begin{Assumption}\label{Assumption_LUSC_V}
    We assume that $V$ is left-upper-semicontinuous (LUSC) along sequences of Bermudan stopping times (which are almost surely finite), that is, if $(\tau_k)$ $\in$ $\Theta^{\N}$ and $\tau$ $\in$ $\Theta$ are such that  $\tau_k$ $\uparrow$ $\tau$, then 
    \begin{equation}\label{Eq_LUSC_V}
         \limsup_{k \rightarrow +\infty} V(\tau_k) \leq V(\tau).
    \end{equation}
\end{Assumption}

\begin{Assumption}\label{Assumption_LUSC_rho}
 We assume that $\rho$ is left-upper-semicontinuous (LUSC) along sequences of Bermudan stopping times (which are almost surely finite), that is, 
    \begin{equation}\label{Eq_LUSC_rho}
         \limsupn \rho_{\theta_k, \tau_n}[\phi(\tau_n)] \leq \rho_{\theta_k, \tau}[\limsupn \phi(\tau_n)].
    \end{equation}
    for each non-decreasing sequence $(\tau_n)$ $\in$ $\Theta_{\theta_k}^{\N}$ and for each $\tau$ $\in$ $\Theta_{\theta_k}$ such that  $\tau_n$ $\uparrow$ $\tau$.
\end{Assumption}

\begin{lemma}\label{Lemma_V_martingale_on_stochastic_interval}
     Let $\nu^{\star}$ be a \textbf{(finite)} stopping time in $\Theta_{\theta_k}$ such that $\nu^{\star} \leq \nu_k$. 
    Then, the value family $V$ is a \emph{$(\Theta, \rho)$}-martingale on $[\theta_k,\nu^{\star}]$.
\end{lemma}

\begin{corollary}\label{corollary_V_martingale_on_interval}
As a consequence of Lemma \ref{Lemma_V_martingale_on_stochastic_interval}, we have that: \\ If $\nu_k$  is a.s. finite, then 
   the value family $V$ is a \emph{$(\Theta, \rho)$}-martingale
    on $[\theta_k,\nu_k]$. 
\end{corollary}

\begin{proof}[Proof of Lemma \ref{Lemma_V_martingale_on_stochastic_interval}]
The proof follows the same type of arguments as the ones presented in \cite{Grigorova-5}. 
\\ As $V$ satisfies the DPP in the case of infinite horizon (cf. Theorem \ref{Theorem_DPP_infinite_horizon}), we apply Lemma 2.25 and Lemma 2.26 in \cite{Grigorova-5} with $\phi=V$ and with $\nu^*$ instead of $\tilde{\nu}_k$, which we can do as $\nu^*$ is finite a.s. and as $\nu^* \leq \nu_k$. 
For each $l \in \N$,
\[
V(\theta_l \wedge \nu^*)=\rho_{\theta_l \wedge \nu^*,\theta_{l+1} \wedge \nu^*}[V(\theta_{l+1} \wedge \nu^*)].
\]
Similarly, by following the reasoning by induction from Step 1 of the proof of Theorem 2.19 in \cite{Grigorova-5} (applied with $\nu^*$ instead of $\nu_k$), we get: for each $n \in \N$,
\[
\rho_{\theta_l \wedge \nu^*,\theta_{l+n} \wedge \nu^*}[V(\theta_{l+n} \wedge \nu^*)] =
V(\theta_l \wedge \nu^*).
\]
Hence, the sequence of random variables $(\rho_{\theta_{l} \wedge \nu^*,\theta_{l+n} \wedge \nu^*}[V(\theta_{l+n} \wedge \nu^*)])_{n \in \N}$ does not depend on $n$ and is constantly equal to the random variable $V(\theta_l \wedge \nu^*)$.
\\
As in \cite{Grigorova-5}, we use that $V$ is LUSC along the sequence $(\theta_{l+n} \wedge \nu^*)_{n \in \N}$ (by Assumption \ref{Assumption_LUSC_V}), and that $\rho$ is LUSC along the sequence $(\theta_{l+n} \wedge \nu^*)_{n \in \N}$ on $\rho$  (by Assumption \ref{Assumption_LUSC_rho}), and the monotonicity of $\rho$, 
to get:
\[
V(\theta_l \wedge \nu^*) \leq
\rho_{\theta_l \wedge \nu^*, \nu^*}[V(\nu^*)]. 
\]
The converse inequality follows from the \emph{$(\Theta, \rho)$}-supermartingale property of $V$ (cf. Theorem \ref{Theorem_supermartingale_Snellenvelope}). 
Hence, we have the equality:
\begin{equation}\label{eq_25}
V(\theta_l \wedge \nu^*) =
\rho_{\theta_l \wedge \nu^*, \nu^*}[V(\nu^*)]. 
\end{equation}
Let $\sigma \in \Theta_{\theta_k}$. 
Then $\sigma$ is of the form $\sigma=\sum_{l \geq k} \theta_l \I_{A_l}$, where $(A_l)_{l \in \N}$ form a partition of $\Omega$, and $A_l$ is $\cf_{\theta_l}$-measurable, for each $l$.
Hence, 
\[
V(\sigma \wedge \nu^*) =
\sum_{l \geq k} V(\theta_l \wedge \nu^*) \I_{A_l} =
\sum_{l \geq k} \rho_{\theta_l \wedge \nu^*, \nu^*}[V(\nu^*)] \, \I_{A_l}
= \rho_{\sigma \wedge \nu^*, \nu^*}[V(\nu^*)],
\]
where we have used \eqref{eq_25} and the admissibility of $V$. 
\\ Hence, $V$ is a \emph{$(\Theta, \rho)$}-martingale on $[\theta_k,\nu^*]$ by Lemma \ref{Lemma_phi_martingale_on_stochastic_interval} in the Appendix.
\end{proof}
A version of the following lemma has been proved in \cite{Grigorova-5} (cf. Lemma 2.28) in the context of problems with finite time horizon.
\begin{lemma} \label{Lemma_assumption_on_rho}
  Let $\rho$ satisfy the properties of knowledge preservation (iii), monotonicity (iv) and consistency (v). 
  \\ If $\rho$ does not depend on the second index 
   (that is, $\rho_{S,\tau}[\cdot]=\rho_{S}[\cdot]$), 
  we \textbf{assume} that: if $\eta \in L^+_0(\cf_{\infty})$, then
  \begin{equation} \label{eq_8}
   \limsup_{n \rightarrow +\infty} \rho_{\theta_n} [\eta] = \eta.
  \end{equation}
   If $\rho$ depends on both indices, 
  we \textbf{assume} that: if $\eta \in L^+_0(\cf_{\infty})$, if 
  $(\tau_k)$ $\in \Theta^{\mathbb{N}}$ non-decreasing such that $\lim_{k \rightarrow +\infty} \tau_k  = +\infty$, then
  \begin{equation} \label{eq_9}
  \limsup_{n \rightarrow +\infty} \limsup_{k \rightarrow +\infty} \rho_{\theta_n, \tau_k} [\eta] =  \eta.
  \end{equation}
When these conditions on $\rho$ are satisfied, we have:
   \begin{equation} \label{eq_limsupV}
      \limsupn V(\theta_n) = \limsupn \xi(\theta_n).
  \end{equation}
\end{lemma}
\begin{Remark}
  In the case where $\rho_{S,\tau}[\cdot]=\rho_{S}[\cdot]=\E[\cdot|\cf_{S}]$, the condition \\ $\limsupn \E[\eta|\cf_{\theta_n}] = \eta$, for each $\eta \in L^+_0(\cf_{\infty})$, 
   is trivially true. 
   \\
   In the case where $\rho_{S,\tau}[\cdot]$ is the $g$-conditional expectation (on infinite horizon) from 
   Section \ref{Example_1}, we have:  $p=2$, $\eta \in L^2$ and 
   $\rho_{S,\tau}[\cdot]=\rho_S[\cdot]=\ce_S^g[\cdot]$,
   and 
   $\limsupn \ce_{\theta_n}^g[\eta] = \eta$
follows by Remark \ref{remark_limY_t} (cf. Section \ref{Example_1}). 
\end{Remark}

\begin{proof}[Proof of Lemma \ref{Lemma_assumption_on_rho}]
    For each $m \in \N$, for each $\tau \in \Theta_{\theta_m}$, $\xi(\tau) \leq \sup_{p \geq m} \xi(\theta_p)$.
    \\ Indeed, for each $\tau \in \Theta_{\theta_m}$, we have
    \begin{equation}\label{eq_10}
    \xi(\tau)
    =\sum_{l \geq m} \xi(\theta_l) \I_{A_l} 
    \leq  \sup_{p \geq m} \xi(\theta_p).
     \end{equation}
    For convenience, let us set $\eta^{(m)}:=\eta:= \sup_{p \geq m} \xi(\theta_p) $.
    As $\xi$ is a non-negative family, $\eta$ is a non-negative random variable in $L^{0}_{+}$.
   \\ For each $n\geq m$ (as $\Theta_{\theta_n} \subseteq \Theta_{\theta_m}$),
    \[
    V(\theta_n)
   = \esssup_{\tau\in \Theta_{\theta_n}} \rho_{\theta_n,\tau}[\xi(\tau)]
   \leq \esssup_{\tau\in \Theta_{\theta_n}} \rho_{\theta_n,\tau}[\eta],
    \]
    where we have used the monotonicity of $\rho_{\theta_n,\tau}$ and \eqref{eq_10} for the inequality.
    \\ Let us focus on the optimal stopping problem on the RHS of the above inequality. 
    The family $\{\rho_{\theta_n,\tau}[\eta]:\tau \in \Theta_{\theta_n}\}$ is stable by maximization. 
    Indeed, let $\tau_1 \in \Theta_{\theta_n}$ and $\tau_2 \in \Theta_{\theta_n}$. 
    Set $A:=\{ \rho_{\theta_n,\tau_2}[\eta] \leq \rho_{\theta_n,\tau_1}[\eta] \}$ and let $\tau:= \tau_1 \I_A + \tau_2\I_{A^c}$. 
    We have: $A \in \cf_{\theta_n}$, 
    and $\tau \in \Theta_{\theta_n}$ (cf. Remark \ref{Remark_set_Theta}). 
    Also, $\tau = \tau_1$ on $A$, and $\tau=\tau_2$ on $A^c$. 
    Hence, by the "generalized" zero-one law property (vi), we get:
    \[ \begin{split}
          \rho_{\theta_n,\tau}[\eta] 
    =\I_A \rho_{\theta_n,\tau}[\eta] + \I_{A^c} \rho_{\theta_n,\tau}[\eta]
    &=\I_A \rho_{\theta_n,\tau_1}[\eta] + \I_{A^c} \rho_{\theta_n,\tau_2}[\eta]
    \\ &=\max(\rho_{\theta_n,\tau_1}[\eta], \rho_{\theta_n,\tau_2}[\eta]).
    \end{split}
  \]
  Hence, the family $\{\rho_{\theta_n,\tau}[\eta]:\tau \in \Theta_{\theta_n}\}$ is stable by maximization. 
  \\
 Hence, for each $n \in \N$, there exists a sequence $(\tau_k) \in \Theta_{\theta_n}^\N$ such that the following maximizing sequence property holds: 
 \[
  \esssup_{\tau\in \Theta_{\theta_n}} \rho_{\theta_n,\tau}[\eta] = \lim_{k \rightarrow +\infty} \uparrow \rho_{\theta_n,\tau_k}[\eta].
 \]
  Hence, for each $n \in \N$, there exists $(\tau_k) \in \Theta_{\theta_n}^\N$ such that 
  \[
    V(\theta_n)
   \leq \lim_{k \rightarrow +\infty} \uparrow \rho_{\theta_n,\tau_k}[\eta].
    \]
   Hence, 
$
     \limsup_{n \rightarrow +\infty}  V(\theta_n)
       \leq  \limsup_{n \rightarrow +\infty}  \lim_{k \rightarrow +\infty} \uparrow \rho_{\theta_n,\tau_k}[\eta]. 
  $
  \\ \textbf{By assumption \eqref{eq_9}} on $\rho$, we have: \[ \limsup_{n \rightarrow +\infty} \lim_{k \rightarrow +\infty} \uparrow \rho_{\theta_n, \tau_k} [\eta] = \limsup_{n \rightarrow +\infty} \limsup_{k \rightarrow +\infty} \rho_{\theta_n, \tau_k} [\eta] = \eta.\] 
  Hence, we obtain: for each $m \in \mathbb{N}$,
\[ 
     \limsup_{n \rightarrow +\infty}  V(\theta_n)
       \leq  
       \eta = \sup_{p \geq m} \xi(\theta_p).
  \]
    Hence, by passing to the limit when $m \rightarrow +\infty$, we get:
       \[
     \limsup_{n \rightarrow +\infty}  V(\theta_n)
   \leq\limsup_{m \rightarrow +\infty} \xi(\theta_m),
    \]
   For the converse inequality, by Remark \ref{TEXRemark33}, Statement 1, if $\rho$ satisfies the assumption of knowledge preservation (property (iii)), then, for each $n\in\N$, $V(\theta_n) \geq \xi(\theta_n)$. Hence, $  \limsup_{n \rightarrow +\infty} V(\theta_n) \geq   \limsup_{n \rightarrow +\infty}\xi(\theta_n)$.
   Thus finishes the proof of the lemma.
\end{proof}

\begin{Remark}
    As a consequence of Lemma \ref{Lemma_assumption_on_rho}, under assumption \eqref{eq_9}, we have: \\
    If the pay-off family $\xi=(\xi(\tau):\tau \in \Theta)$ is such that $\limsup_{n \rightarrow +\infty} \xi(\theta_n) <+\infty$,
    then $ \limsup_{n \rightarrow +\infty}  V(\theta_n) <+\infty$.
\end{Remark}

We are now ready to prove Theorem \ref{Theorem_optimal_stopping_times}.
\begin{proof}[Proof of Theorem \ref{Theorem_optimal_stopping_times}]
     Let us show Statement 1. 
    \\ Let $\varepsilon >0$ and let $k \in \mathbb{N}$. 
     Let $\nu_k^{\varepsilon}$ be the stopping time defined in Eq. \eqref{Defn.epsilon_stopping_time}.
     \\ By Lemma \ref{Lemma_assumption_on_rho}, $   \limsupn V(\theta_n) = \limsupn \xi(\theta_n)$. Hence, as $\limsupn \xi(\theta_n)<+\infty$ by assumption, then $   \limsupn V(\theta_n) < +\infty$. Hence, in particular, on the event $\{\nu_k^{\varepsilon}=+\infty\},$ we have:
    \begin{equation} \label{Statment_V_finite}
      \limsupn V(\theta_n) =\limsupn \xi(\theta_n)<+\infty.
    \end{equation}
    On the other hand, on the same event $\{\nu_k^{\varepsilon}=+\infty\}$, we have: for each $\tau \in \Theta_{\theta_k}$, 
    \[
     V(\tau) \geq \xi(\tau)+\varepsilon \quad a.s.,
    \]
    by the definitions of $\nu_k^{\varepsilon}$ and of $A_k^{\varepsilon}$.
    \\ In particular, on $\{\nu_k^{\varepsilon}=+\infty\}$, for $n \geq k$, 
      $ V(\theta_n) \geq \xi(\theta_n)+\varepsilon$ a.s.
  \\
    Hence, on  $\{\nu_k^{\varepsilon}=+\infty\}$, 
    \begin{equation}
         \limsupn V(\theta_n) \geq \limsupn \xi(\theta_n)+\varepsilon \quad a.s.   \end{equation}
  But, $  \limsupn V(\theta_n) =\limsupn \xi(\theta_n)< + \infty$ by Statement \eqref{Statment_V_finite}. 
   \\ Hence, the event $\{\nu_k^{\varepsilon}=+\infty\}$ is $P$-negligible (that is, with $P$-probability zero). 
    \\
     Therefore, the set $A_k^{\varepsilon}$ is non-empty. 
     Moreover, it is stable by pairwise minimization. Hence, there exists a non-increasing sequence $(\tau_n)$ in $A_k^{\varepsilon}$ such that $ \tau_n \downarrow  \nu_k^{\varepsilon}$ a.s.
    In particular, $\nu_k^{\varepsilon}$ is a stopping time, $\nu_k^{\varepsilon} \geq \theta_k$ a.s., and $\nu_k^{\varepsilon} \in \Tk$ by stability of $\Tk$ when passing to a non-increasing limit(cf. Remark \ref{Remark_closed_under_limit}, Statement a).
    \\ 
    Moreover, by the definition of $\nu_k^{\varepsilon}$, by Remark \ref{Remark_RC},  and the fact that $\nu_k^{\varepsilon} < +\infty$ a.s.,
    \[
    V(\nu_k^{\varepsilon}) \leq \xi(\nu_k^{\varepsilon}) + \varepsilon.
    \]
Hence, by the monotonicity of $\rho$, 
  \[
   \rho_{\theta_k, \nu_k^{\varepsilon}}[V(\nu_k^{\varepsilon})] \leq \rho_{\theta_k, \nu_k^{\varepsilon}}[\xi(\nu_k^{\varepsilon}) + \varepsilon].
    \]
    By Assumption \ref{Assumption_rho_C} on the non-linear operator $\rho$, we get
     \[
\rho_{\theta_k, \nu_k^{\varepsilon}}[\xi(\nu_k^{\varepsilon}) + \varepsilon]
\leq \rho_{\theta_k, \nu_k^{\varepsilon}}[\xi(\nu_k^{\varepsilon})]+C\varepsilon,
    \]
    where $C>0$ (depending on $\rho$ only) is the constant from Assumption \ref{Assumption_rho_C}. 
    \\
    By using the \emph{$(\Theta, \rho)$}-martingale property of $V$ on $[\theta_k, \nu_k^{\varepsilon}]$ from Lemma \ref{Lemma_V_martingale_on_stochastic_interval} (which can be applied as $\nu_k^{\varepsilon}<+\infty$ and $\nu_k^{\varepsilon} \leq \nu_k$),   
     we get:
     $V(\theta_k)= \rho_{\theta_k, \nu_k^{\varepsilon}}[V(\nu_k^{\varepsilon})] $. 
   Hence,  $ V(\theta_k) = \rho_{\theta_k, \nu_k^{\varepsilon}}[V(\nu_k^{\varepsilon})]
         \leq \rho_{\theta_k, \nu_k^{\varepsilon}}[\xi(\nu_k^{\varepsilon})] +C\varepsilon. 
    $ 
    Hence, by setting $\tilde{\varepsilon} =C\varepsilon$, we have shown that $\nu_k^{\varepsilon}$ is $\tilde{\varepsilon} $-optimal. 
    This proves Statement 1.   

\noindent
Let us show Statement 2. \\
The proof of the second statement in the theorem is a consequence of Corollary \ref{corollary_V_martingale_on_interval}, of the optimality criterion from Proposition \ref{prop_optimality_criterion} , and Remark \ref{Remark_v_k_finite}. Indeed, if $\nu_k$ is finite a.s., then by Corollary \ref{corollary_V_martingale_on_interval}, V is a \emph{$(\Theta, \rho)$}-martingale on $[\theta_k, \nu_k]$. 
Moreover, by Remark \ref{Remark_v_k_finite}, $ V(\nu_k)=\xi(\nu_k)$.
Hence, by the optimality criterion (cf. Proposition \ref{prop_optimality_criterion}), we get: $\nu_k$ is optimal for the problem \eqref{optimal_stopping_problem_k}. 
\\ The converse implication in Statement 2 of the theorem is trivial. Indeed, if $\nu_k$ is optimal for the problem \eqref{optimal_stopping_problem_k}, then $\nu_k \in \Tk$, and in particular, 
$\nu_k < +\infty$ a.s..
\\
\end{proof}
\section{Doob's type convergence for non-negative \emph{$(\Theta, \rho)$}-supermartingales}\label{sec_Doob}

As we deal in this work with non-negative \emph{$(\Theta, \rho)$}-supermartingales $V$ on an \textbf{infinite time horizon}, it is natural to ask ourselves whether we can establish a Doob-type convergence result showing that $\limn V(\theta_n)$ exists in the a.s. sense. 
This section provides a positive answer to this question under some additional assumptions on the operators $\rho$. 

\noindent
We first state and show an important intermediate result. 
For this, we assume dependence on the first index only, that is, $\rho_{S, \tau}[\cdot]=\rho_{S}[\cdot]$. We assume the following additional properties on $\rho$: 
\begin{compactenum}
    \item [(viii)] \emph{("usual" zero-one law) }$\; I_A\rho_{S}[\eta] =\rho_{S}[I_A\eta],$ for all $A\in\cf_S$.  
    \item[(ix)] \emph{(positive homogeneity)} $\rho_S[b\eta]=b\rho_S[\eta]$, for $b>0$.
    \item[(x)] $\rho_{0} [\I_{A} ] =0  $ implies $P(A)=0$. 
\end{compactenum}

\begin{lemma}[Upcrossing]\label{Upcrossing_lemma}
Let $\rho$, depending on the first index only, satisfy properties (i),(iv),(v),(vii),(viii),(ix),(x), and the normalisation property $\rho_0[1]=1$. 
\\
   Let $X$ be a \textbf{non-negative} \emph{$(\Theta, \rho)$}-supermartingale.
    Let $0<a<b<+\infty$ ( $a,b \in \mathbb{Q}$).
    We define, by induction, the following stopping times: 
   $ T_0:=0 $
 and for all $k\in \N^*$, 
 \begin{align*}
S_k&:=\inf\{n \geq T_{k-1}:X_n \leq a\},  \\
T_k&:=\inf\{n \geq S_{k}:X_n \geq b\},
\end{align*}  
 with the convention that $inf \emptyset = + \infty$.
 \\ 
 Let $N_{\infty}([a,b])$ be the number of upcrossings of the interval $[a,b]$. 
  Let $N_{n}([a,b])$ be the number of upcrossings of the interval $[a,b]$ before time $n$.
  \\ Then, 
   \begin{compactenum}
   \item[(i)] $\rho_{S_k \wedge n}[\I_{\{T_k \leq n\}}]\leq (\frac{a}{b})\I_{\{S_k \leq n\}}$.
   \item[(ii)] $\rho_0[\I_{\{T_k<+\infty\}}]\leq (\frac{a}{b})^k$.
    \item[(iii)] 
  $P(N_{\infty}([a,b])< +\infty ) =1$.
   \end{compactenum}
\end{lemma}
\begin{proof}
    We have: 
    \[
    0=T_0 \leq S_1 \leq T_1 \leq \cdots \leq  S_k \leq T_k \leq \cdots. 
    \]
    We thus have: $\{T_k\leq n \} \subset \{S_k\leq n \} \subset  \{T_{k-1}\leq n \}$ (as $ T_{k-1} \leq S_k \leq T_k $).\\
    We also note that: 
    \[
   \{ N_{\infty}([a,b]) \geq k \} = \{T_k < +\infty\}.
    \]
    As $X$ is a \emph{$(\Theta, \rho)$}-supermartingale,
    \begin{equation} \label{eq_19}   
  \rho_{S_k \wedge n} [ X_{T_k \wedge n}] \leq X_{S_k \wedge n}.
    \end{equation}
  Then, 
    \[
 \I_{\{S_k\leq n \}} \rho_{S_k \wedge n} [ X_{T_k \wedge n}] \leq  \I_{\{S_k\leq n \}} X_{S_k \wedge n}.
    \]
    Hence,  as $ X_{S_k \wedge n} \leq a$ on $\{S_k\leq n \}$,
    \[
 \I_{\{S_k\leq n \}} \rho_{S_k \wedge n} [ X_{T_k \wedge n}] \leq  \I_{\{S_k\leq n \}} X_{S_k \wedge n}
\leq  \I_{\{S_k\leq n \}} a.
     \]
  On the other hand, 
      \[\begin{split}
 \I_{\{S_k\leq n \}} \rho_{S_k \wedge n} [ X_{T_k \wedge n}] &=  \I_{\{S_k\leq n \}} \rho_{S_k \wedge n} [ X_{T_k \wedge n}\I_{\{T_k\leq n \}} +  X_{T_k \wedge n}\I_{\{T_k > n \}} ] 
 \\& \geq  \I_{\{S_k\leq n \}} \rho_{S_k \wedge n} [ X_{T_k \wedge n} \I_{\{T_k\leq n \}} ],
      \end{split} \]
      as $X$ is non-negative and $\rho_{S_k\wedge n}$ is monotone.
      \\   Hence,
      \begin{equation}\label{eq_22}
           \I_{\{S_k\leq n \}} \rho_{S_k \wedge n} [ X_{T_k \wedge n} \I_{\{T_k\leq n \}} ]
      \leq  \I_{\{S_k\leq n \}} a.
      \end{equation}
     By the "usual" zero-one law property (property (viii)), 
      \begin{equation}\label{eq_23}
       \I_{\{S_k\leq n \}} \rho_{S_k \wedge n} [ X_{T_k \wedge n} \I_{\{T_k\leq n \}} ]
      =   \rho_{S_k \wedge n} [ X_{T_k \wedge n} \I_{\{S_k\leq n \}} \I_{\{T_k\leq n \}} ]
       =   \rho_{S_k \wedge n} [ X_{T_k \wedge n} \I_{\{T_k\leq n \}} ],
       \end{equation}
      where we have used that $\{S_k\leq n , T_k \leq n \}=\{T_k \leq n \} $. 
      \\  On the other hand, 
      \[
       X_{T_k \wedge n} \I_{\{T_k\leq n \}} \geq b  \I_{\{T_k\leq n \}}. 
      \]
  Using this, Eqs. \eqref{eq_22} and \eqref{eq_23}, and the monotonicity of $\rho_{S_k \wedge n}$, we get:
     \[
      \rho_{S_k \wedge n} [ b \I_{\{T_k\leq n \}} ] \leq 
      \rho_{S_k \wedge n} [ X_{T_k \wedge n} \I_{\{T_k\leq n \}}] 
      \leq a \I_{\{S_k\leq n \}}.
     \]
       Using now the positive homogeneity of $\rho$ (as $b >0$), we get: 
          \[
      \rho_{S_k \wedge n} [\I_{\{T_k\leq n \}} ] \leq \frac{a}{b} \I_{\{S_k\leq n \}},
     \]
     which is Statement (i).
    \\ From this, we have:  
       \[
      \rho_{S_k \wedge n} [\I_{\{T_k\leq n \}} ] \leq  \left(\frac{a}{b}\right) \I_{\{S_k\leq n \}} \leq \left(\frac{a}{b}\right) \I_{\{T_{k-1}\leq n \}}, \quad \text{as} \quad  T_{k-1} \leq S_k.
     \] 
     By using the monotonicity, the consistency, and the positive homogeneity of $\rho$, we get:
     \begin{equation}\label{eq_18}
      \rho_{0} [\I_{\{T_k\leq n \}} ]  \leq \left(\frac{a}{b}\right) \rho_{0} [\I_{\{T_{k-1}\leq n \}}].
        \end{equation}
      We will pass to the limit when $n$ tends to $+\infty$. 
      We have: $\I_{\{T_k\leq n \}}$ is non-decreasing in $n$ and its limit is $ \I_{\{T_k < +\infty \}}$.  Also, $ \I_{\{T_{k-1}\leq n \}}$ is non-decreasing in $n$ and its limit is $ \I_{\{T_{k-1} < +\infty \}}$. 
       \\ By using this and the continuity from below (cf. Remark \ref{Remark_CB}), we get: 
      \[
         \rho_{0} [\I_{\{T_k< +\infty \}} ]  \leq  \left(\frac{a}{b}\right) \rho_{0} [\I_{\{T_{k-1}< +\infty \}}].
       \]
        By iterating over $k$ and using 
        the normalisation property ($\rho_0[1]=1$), we get: 
        \[
         \rho_{0} [\I_{\{T_2< +\infty \}} ]  \leq  \left(\frac{a}{b}\right) \rho_{0} [\I_{\{T_{1}< +\infty \}}] 
         \leq  \left(\frac{a}{b}\right)^2 \rho_{0} [\I_{\{T_{0}< +\infty \}}] 
         = \left(\frac{a}{b}\right)^2 \rho_{0} [1] = \left(\frac{a}{b}\right)^2 1,
       \]
       and, more generally, 
        \[
         \rho_{0} [\I_{\{T_k< +\infty \}} ]   
         \leq  \left(\frac{a}{b}\right)^k \rho_{0} [\I_{\{T_{0}< +\infty \}}] 
         = \left(\frac{a}{b}\right)^k \rho_{0} [1] = \left(\frac{a}{b}\right)^k 1,
       \]
       which is Statement (ii). 
      \\ Hence, 
         \[
         \rho_{0} [\I_{\{ N_{\infty}([a,b]) \geq k \} } ]  \leq  \left(\frac{a}{b}\right)^k.
       \]
       To prove Statement (iii),
       we will pass to the limit when $k \rightarrow +\infty$.
        We set: 
       for $k\geq 0$, $A_k:={\{ N_{\infty}([a,b]) \geq k \} } $. 
       We have: $A_{k+1} \subset A_k$, for each $k$, and $(\I_{A_k})$ is non-increasing in $k$. 
       Hence, by the monotonicity of of $\rho_0$, $(\rho_0[\I_{A_k}])$ is a non-increasing sequence (in $k$). 
       We thus get:  
      \[
         0 \leq \lim_{k \rightarrow +\infty} \downarrow \rho_{0} [\I_{\{ N_{\infty}([a,b]) \geq k \} } ]=   \lim_{k \rightarrow +\infty} \downarrow \rho_{0} [\I_{A_k} ]\leq  \lim_{k \rightarrow +\infty} \left(\frac{a}{b}\right)^k = 0, \quad \text{as} \quad b>a>0.
       \] 
        We have, by the monotonicity of $\rho_0$,
       \[
    \rho_{0} [\I_{\{ N_{\infty}([a,b]) \geq +\infty \} } ] \leq \rho_{0} [\I_{\{ N_{\infty}([a,b]) \geq k \} } ], \quad \text{for any} \quad k \in \N.
       \]
      Hence, 
       \[
       \rho_{0} [\I_{\{ N_{\infty}([a,b]) \geq +\infty \} } ] \leq  \lim_{k \rightarrow +\infty} \downarrow \rho_{0} [\I_{\{ N_{\infty}([a,b]) \geq k \} } ] = 0.
       \]
       This implies, by using property (x), that 
       \[
     P(N_{\infty}([a,b])\geq +\infty ) =0,
       \]
        which is Statement (iii).
\end{proof}

\begin{theorem}[Doob-type convergence theorem for \emph{$(\Theta, \rho)$}-supermartingale]\label{Doob's_type_convergence_theorem}
Let $\rho$ satisfy the properties from the upcrossing lemma (Lemma \ref{Upcrossing_lemma}). 
   Let $M=(M(\tau))_{\tau \in \Theta_{\theta_k}}$ be a non-negative \emph{$(\Theta, \rho)$}-supermartingale. 
   Then, there exists a (non-negative) random variable $M_{\infty}$ such that 
   \[
   M_{\infty}= \lim_{n \rightarrow \infty} M(\theta_n) \quad a.s.  
   \]
\end{theorem}
\begin{Remark}
We place ourselves under the assumptions of Theorem \ref{Doob's_type_convergence_theorem}.
We assume moreover that $M(0) < \infty$ a.s. and that $\rho_0$ satisfies \textbf{the strong Fatou property}, that is, 
$\rho_{0}[\eta] \leq \liminf_{n \to +\infty} \rho_{0}[\eta_{n}]$, for $(\eta_{n}), \eta$ such that $\lim_{n \to +\infty}\eta_{n} = \eta$ a.s.
Then,  $M_{\infty}$ is finite a.s. \\
  Indeed, let $A_k := \{M_{\infty} \geq k\}$, for $k \geq 1$. 
 We have $A_{k+1} \subset A_k$, for each $k$.
 Moreover, as $M$ is non-negative, we have: $ \I_{ \{M_{\infty} \geq k \} } \leq \frac{(M_{\infty} \wedge k)}{k} \leq \frac{M_{\infty}}{k} $. 
 By the monotonicity of $\rho_0$, we get: 
 \begin{equation} \label{eq_21}
        \rho_{0} [\I_{\{M_{\infty} \geq k \} } ]  \leq  \rho_{0} \left[\frac{M_{\infty} \wedge k}{k}\right]\leq  \frac{1}{k} \rho_{0} \left[M_{\infty}\right].
    \end{equation}
       where we have used also the positive homogeneity of $\rho_0$ in the last inequality.
       \\
       We will show that \begin{equation}\label{Eq_24}
           \rho_{0} \left[M_{\infty}\right]<+\infty.
       \end{equation}
       Once, this is proven, we can pass to the limit when $k \rightarrow +\infty$, to get: 
       \[
       0\leq \lim_{k \rightarrow +\infty} \downarrow \rho_{0} [\I_{\{M_{\infty} \geq k \} } ] \leq  \lim_{k \rightarrow +\infty} \left (\frac{1}{k} \rho_{0} \left[M_{\infty}\right] \right) =0
       \]
       Hence, as $\rho_{0} [\I_{\{M_{\infty} \geq \infty \} } ] \leq \rho_{0} [\I_{\{M_{\infty} \geq k \} } ]$, for each $k \geq 1$, 
       we conclude that $\rho_{0} [\I_{\{M_{\infty} \geq \infty \} } ] =0$.
       This implies, by using property (x), that 
     $P(M_{\infty} = +\infty ) =0$.
    Thus, $M_{\infty}$ is finite. 
       It remains to show \eqref{Eq_24}. 
       \\
       As $M$ is a \emph{$(\Theta, \rho)$}-supermartingale, for each $n$,
          $\rho_{0} \left[ M(\theta_n) \right]\leq  M(\theta_0)$. %= M(0)
       Hence, 
       $
     \sup_{n}   \rho_{0} \left[ M(\theta_n) \right] 
           \leq    M(\theta_0) 
       $. 
     \\ On the other hand, 
      we have: 
       \[  \rho_{0} \left[M_{\infty}  \right]   
            =   \rho_{0} \left[ \lim_{n \rightarrow \infty} M(\theta_n) \right]   
             \leq  \liminf_{n \rightarrow \infty}   \rho_{0} \left[ M(\theta_n) \right]   
             \leq  \sup_{n}   \rho_{0} \left[ M(\theta_n) \right], 
       \]
       where we have used the \textbf{strong Fatou property of $\rho_0$} in the first inequality.
       \\  By \textbf{the assumption $M(0)< +\infty$}, and the \emph{$(\Theta, \rho)$}-supermartingale property of $M$, we thus get: $ \rho_{0} \left[M_{\infty}\right]  < +\infty$.
\end{Remark}
\begin{proof}[Proof of Theorem \ref{Doob's_type_convergence_theorem}]
The proof of this theorem follows from the upcrossing lemma (Lemma \ref{Upcrossing_lemma}), combined with a standard reasoning from real analysis.
\\
Let $E$ be the set \[E:=\{\omega \in \Omega: 0 \leq \liminf_{n \rightarrow \infty} M(\theta_n)(\omega) < \limsupn M(\theta_n)(\omega) \}.\] 
To prove that the non-negative \emph{$(\Theta, \rho)$}-supermartingale $M$ converges a.s., we will show that $P(E)=0$. 
\\
The set $E$ is equal to $\bigcup_{\substack{a,b \in \mathbb{Q} \\ 0<a<b<+\infty}}  E_{a,b}$, where for each $a,b$ such that $0<a<b<+\infty$,
\[
E_{a,b}:=\{\omega \in \Omega: \liminf_{n \rightarrow \infty} M(\theta_n)(\omega) <a<b< \limsupn M(\theta_n)(\omega)\}.
\]
On the other hand, for each $a,b$ such that $0<a<b<+\infty$, %we fixed a,b
$E_{a,b} \subset \{N_{\infty}([a,b])= +\infty\}$.
Hence, by the upcrossing lemma (Lemma \ref{Upcrossing_lemma}), %statement (iii) 
we have: \[P(E_{a,b}) \leq P( \{N_{\infty}([a,b])= +\infty\})=0.\]
Hence, by countable union,  $ P(E)=P\left(\bigcup_{\substack{a,b \in \mathbb{Q} \\ 0<a<b<+\infty}}  E_{a,b}\right)=0$, which is the desired result.
\end{proof}

\section{Example: Infinite horizon BSDEs with driver $g$ and $\ce^g$-evaluation}\label{Example_1}

\cite{chen2000infinite} have considered \textbf{infinite horizon} BSDEs in the Brownian framework with square-integrable terminal condition (at $T=+\infty$) given by $\eta \in L^2(\cal F_{\infty})$.
\\ Let $(\mathcal{F}_{t})_{t\geq 0}$ be the (augmented) natural filtration generated by a one-dimensional Brownian motion $(W_t)_{t\geq 0}$.
\\
\cite{chen2000infinite} considered the following infinite horizon BSDE with terminal condition $\eta \in L^2(\cal F_{\infty})$:
\begin{equation} \label{eq_BSDE}
    Y_t = \eta + \int_t^{+\infty} g(s,Y_s,Z_s) ds -  \int_t^{+\infty} Z_s dW_s , \quad {\rm for} \quad t \geq 0,
\end{equation}
where $g:(\Omega,\R_+,\R,\R)\rightarrow \R$ is a driver satisfying the two conditions: 
\begin{compactenum}
    \item[(H1)] For each $(y,z) \in \R \times \R$,  $g(\cdot, \cdot,y,z)$ is progressively measurable and such that  
   $\E\left[\left(\int_0^{+\infty}|g(s,y,z)|ds\right)^2\right]<+\infty$
    \item[(H2)] There exist two positive deterministic functions $u:\R_+ \rightarrow \R_+$ and $v:\R_+ \rightarrow \R_+$ such that 
    $\int_0^{+\infty}|u(s)|^2ds<+\infty$ and  
    $\int_0^{+\infty}|v(s)|ds<+\infty$, and 
    \[
  |g(t,y_1,z_1) -  g(t,y_2,z_2)| \leq v(t) |y_1-y_2| +u(t) |z_1-z_2|, 
    \]
    for each $t \geq 0$, for each $(y_1,z_1) \in \R^2$ and each $(y_2,z_2) \in \R^2$.  % MoE 3.  
 
\end{compactenum}
\noindent
\\ If $g$ satisfies the above properties, then we write $g \in \mathscr{L}(u,v)$.

 We use the following notation: 
 \begin{itemize}
     \item $H^2(0,+\infty)$ is the space of predictable stochastic processes $X=(X_t)_{t \geq 0}$ such that 
     \[
     ||X||^2_{H^2} := ||X||^2_{H^2(0,+\infty)} := \E\left(\int_{0}^{+\infty} |X_s|^2 ds\right) < +\infty.
     \]

      \item $S^2(0,+\infty)$ is the space of adapted continuous processes $X=(X_t)_{t \geq 0}$ such that 
     \[
     ||X||^2_{S^2} := ||X||^2_{S^2(0,+\infty)} := \E(\sup_{t \geq 0} |X_t|^2) < +\infty. 
     \]
 \end{itemize}
  \noindent  
 \\ We will denote the BSDE from Eq.\eqref{eq_BSDE} with input $(\eta,g)$ where $\eta \in L^2(\cal F_{\infty})$ and $g \in \mathscr{L}(u,v)$ by ${\rm BSDE}_{\infty} (\eta,g)$.
 \\ \cite{chen2000infinite} have shown: if $g \in \mathscr{L}(u,v)$ and if $\eta \in L^2(\mathcal{F}_{\infty})$, then there exists a unique solution pair $(Y_t,Z_t) \in S^2(0,+\infty) \times H^2(0,+\infty)$ to the ${\rm BSDE}_{\infty}(\eta, g)$ from Eq. \eqref{eq_BSDE}. 
 Moreover, let $\eta_1 \in L^2(\mathcal{F}_{\infty})$, $\eta_2 \in L^2(\mathcal{F}_{\infty})$, let $g \in \mathscr{L}(u,v)$, let $(Y^1,Z^1)$ (resp. $(Y^2,Z^2)$) be the unique solution to the ${\rm BSDE}_{\infty}(\eta_1,g)$  (resp. to the ${\rm BSDE}_{\infty}(\eta_2,g)$). %Theorem 1.3 in \cite{chen2000infinite}
 Then, 
 \[
  ||Y^1-Y^2||_{S^2}^2 + ||Z^1-Z^2||_{H^2}^2 \leq C \E(|\eta_1 - \eta_2|^2) ,
 \]
 where $C>0$.  
 \\
 Furthermore, the following convergence result holds true: 
 \begin{compactenum}
     \item[(C1)] %%Theorem 1.4 in \cite{chen2000infinite}
     Let $\eta_k \in L^2(\mathcal{F}_{\infty})$, for each $k \in \N$, 
     and let $\eta \in L^2(\mathcal{F}_{\infty})$. 
     Let $g \in \mathscr{L}(u,v)$, 
     and let $(Y^k,Z^k)$ be the solution to the ${\rm BSDE}_{\infty} (\eta_k,g)$,  for each $k \in \N$.
     Then, the condition
     \[
    ||\eta_k - \eta||_{L^2(\mathcal{F}_{\infty})}  \xrightarrow[k \rightarrow +\infty]{}   0
     \]
     implies 
      \[
    ||Y^k - Y||_{S^2}  \xrightarrow[k \rightarrow +\infty]{}   0,
     \quad {\rm and} \quad
    ||Z^k - Z||_{H^2}  \xrightarrow[k \rightarrow +\infty]{}   0, 
     \]
     where $(Y,Z)$ is the solution to the ${\rm BSDE}_{\infty} (\eta,g)$. 
     \item[(C2)] %%Corollary 1.5 in \cite{chen2000infinite}
     Let $t > 0$ be a finite horizon. 
     Let $\eta \in L^2(\mathcal{F}_{\infty})$
      and $g \in \mathscr{L}(u,v)$, 
     and let $(Y,Z)$ be the unique solution of the ${\rm BSDE}_{\infty} (\eta,g)$.
     Let $\eta^t:=\E[\eta|\mathcal{F}_t]$.
     Let $(Y^t, Z^t)$ be the unique solution to the BSDE with finite terminal horizon $t$,  and terminal condition $\eta^t$ (at time $t$), where $(Y^t_s)$ is extended for $s>t$, by setting $Y^t_s=\eta^t$, for $s>t$.
     Then, 
       \[
    ||Y^t - Y||_{S^2(0,+\infty)}  \xrightarrow[t \rightarrow +\infty]{}   0,
     \quad {\rm and} \quad
    ||Z^t - Z||_{H^2(0,+\infty)}  \xrightarrow[t \rightarrow +\infty]{}   0.
     \]
     \\
 \end{compactenum}

Some more properties about these BSDEs are established in \cite{chen2000infinite} under the additional assumption (H) on $g$, where:
\\ (H) $g(t,y,0)=0$, for each $(t,y) \in \R_+ \times \R$.

Under (H), they define: \\
The $g$-conditional expectation of $\eta \in L^2(\mathcal{F}_{\infty})$ by
\begin{equation}\label{eq_g_conditional_expectation}
   \ce^g[\eta|\mathcal{F}_{t}] := 
   \ce^g_t[\eta] := 
   Y_t,
\end{equation}
where $(Y_t)_{t \geq 0}$ is the first component of the solution of ${\rm BSDE}_{\infty} (\eta,g)$.
\\ The following properties hold true (for $\eta \in L^2(\mathcal{F}_{\infty})$):
\begin{compactenum}
    \item[(i)] \emph{(usual zero-one law)}  %%Lemma 2.2 in \cite{chen2000infinite}
    $\; \ce_{t}^g[I_A\eta] =I_A\ce_{t}^g[\eta],$ for all $A\in\cf_t$.
    \item[(ii)] \emph{(monotonicity)} %Theorems A.2 and A.3 in \cite{chen2000infinite}
    If  $\eta_1\geq \eta_2$, then
     $\ce_{t}^g[\eta_1]\geq \ce_{t}^g[\eta_2]$, for each $t \geq 0$. 
\item[(iii)] \emph{(strict monotonicity)} 
 If  $\eta_1\geq \eta_2$ and $P(\eta_1 > \eta_2)>0$, then
     $\ce_{t}^g[\eta_1]>\ce_{t}^g[\eta_2]$, for each $t \geq 0$. 
\item[(iv)]  \emph{(knowledge preservation)} %Theorems A.4,(i) in \cite{chen2000infinite}
If $\eta$ is moreover $\cf_t$-measurable, then  $\ce_{t}^g[\eta]=\eta$. 
\item[(v)] \emph{(time consistency)} For all $s, t$ such that $0 \leq s \leq t <+\infty$, $\;\ce_{s}^g[\ce_{t}^g[\eta]]= \ce_{s}^g[\eta]$. 
\item[(vi)] \emph{(continuity with respect to terminal condition)}  %Theorems A.3, (4), in \cite{chen2000infinite}
For $(\eta_k), \eta \in L^2(\mathcal{F}_{\infty})$ such that $  || \eta_k - \eta||_{L^2(\mathcal{F}_{\infty})}  \xrightarrow[k \rightarrow +\infty]{} 0$, then 
\[
  || \ce^g_{\cdot}[\eta_k] - \ce^g_{\cdot}[\eta]||_{S^2}  \xrightarrow[k \rightarrow +\infty]{}  0.
\]
%\item[(vii)] \emph{(ph)}
\end{compactenum}

\begin{Remark}
    If we assume, moreover, that $g$ satisfies: \\ (PH) for $b > 0$, $g(t,by,bz)=bg(t,y,z)$, we conclude from Eq.\eqref{eq_BSDE} and from the existence and uniqueness of the solution (by using a standard argument mimicking the one from the BSDE with a finite horizon) that $\ce^g[\cdot|\cf_t]=\ce^g_{t}[\cdot]$ is positively homogeneous (PH), that is: for $b > 0$, for $\eta \in L^2(\cf_{\infty})$, $\ce^g[b\eta|\cf_t]=b\ce^g[\eta|\cf_t]$.
\end{Remark}
\noindent
\textbf{Financial example} \textbf{(a)} Let us place ourselves in a complete financial market under a Black and Scholes-type model. 
Let us take $r=0$, $\sigma=1$.
In this case, the market price of risk, $\theta(t):=\frac{\mu(t)-r}{\sigma}=\mu(t)$. 
If $\mu(t)$ is a deterministic function such that  $\int_0^{+\infty}|\mu(t)|^2dt<+\infty$, 
then $g$ 
defined by $g(t,y,z)=-\theta(t)z$ satisfies the properties $(H1)$, $(H2)$ and $(H)$.
 Moreover, in this case, for $b>0$, $g(t,by,bz)=bg(t,y,z)=-\theta(t)bz$. 
 Hence, (PH) is also satisfied. 
 \\  \textbf{(b)} Let $r=0$ (as before). 
 Let now $(\mu_t)$ and $(\sigma_t)$ be predictable processes such that $\sigma_t >0$. 
 Let us define the market price of risk, $\theta_t:=\frac{\mu_t}{\sigma_t}$, for each $t \geq 0$. 
 Then, $g(t,y,z):=-\theta_tz$ satisfies the properties $(H1)$, $(H2)$ and $(H)$ if there exists a function $u(t)$ such that $\int_0^{+\infty}|u(t)|^2 dt < +\infty$ and such that: %uniformly in $\omega$
 a.s. for each $t\geq 0$, $|\theta_t| \leq |u(t)|$.
 Moreover, in this case, (PH) is also satisfied.
\begin{Remark}\label{remark_limY_t}
The $g$-conditional expectation $\ce^g_t[\eta]$ 
 satisfies the following property: for $\eta \in L^2(\mathcal{F}_{\infty})$,
   $\lim_{t \rightarrow +\infty} \ce_{t}^g[\eta] = \eta$ in $L^2$ (thus, up to a subsequence,  $\lim_{t \rightarrow +\infty} \ce_{t}^g[\eta] = \eta$ a.s.).
    Indeed, from the BSDE \eqref{eq_BSDE}, we have :
    \[ \begin{split}
    \E |Y_t - \eta|^2 
     &\leq \E \left( \left|\int_t^{+\infty} g(s,Y_s,Z_s) ds\right| + \left|\int_t^{+\infty} Z_s dW_s \right|\right)^2
     \\ &\leq 2 \E\left( \left|\int_t^{+\infty} g(s,Y_s,Z_s) ds\right|^2\right) + 2 \E \left( \left|\int_t^{+\infty} Z_s dW_s\right|^2\right)
      \\ &\leq 2 \E\left( \int_t^{+\infty} \left|g(s,Y_s,Z_s) \right| ds\right)^2 + 2 \E \left(\int_t^{+\infty} Z_s^2 ds\right).
     \end{split} \]
 Since $Z \in H^2$, $\E \left( \int_0^{+\infty} Z_s^2 ds\right) < \infty $. We thus have: 
$ 
\E \left( \int_t^{+\infty} Z_s^2 ds\right) 
 =  \int_t^{+\infty} \E \left( Z_s\right)^2 ds
\xrightarrow[t \rightarrow\infty]{} 0$,
where we have used the dominated convergence theorem. 
\\ 
Moreover, under assumptions (H2) and (H),
we have: 
 \[\begin{split}
\E \left( \int_0^{+\infty} |g(s,Y_s,Z_s)| ds \right)^2
&\leq \E\left(\int_0^{+\infty} v(s)|Y_s|+u(s) |Z_s| ds\right)^2 
% by using (a+b)^2 \leq 2a^2 + 2b^2
\\&\leq 2\E\left(\int_0^{+\infty} v(s)|Y_s| ds\right)^2
+ 2 \E\left(\int_0^{+\infty} u(s)|Z_s| ds\right)^2.
   \end{split}  \]
 Also, by the assumption (H2), we get:
 \[
 \E\left(\int_0^{+\infty} v(s)|Y_s| ds\right)^2 
 \leq 
 \E\left(\int_0^{+\infty} v(s) ds  \, \sup_{t\geq0}|Y_t| \right)^2
  \leq 
 \left(\int_0^{+\infty} v(s) ds  \right)^2  ||Y||_{S^2}^2
 < +\infty,
 \]
 and 
 \[
  \E\left(\int_0^{+\infty} u(s) |Z_s| ds\right)^2 
  \leq \E \left( \int_0^{+\infty} u^2(s)ds \, \int_0^{+\infty} |Z_s|^2ds\right)
  =  \int_0^{+\infty} u^2(s)ds \, ||Z||_{H^2}^2
  < +\infty ,
 \]
 where we have used the Cauchy-Schwarz inequality and assumption (H2).
 \\ Hence, we obtain:
 \[
 \E \left( \int_0^{+\infty} |g(s,Y_s,Z_s)| ds \right)^2 < +\infty.
 \]
 Hence, we have: $\int_0^{+\infty} |g(s,Y_s,Z_s)|ds < +\infty$ a.s., and $\int_t^{+\infty} |g(s,Y_s,Z_s)| ds \xrightarrow[t \rightarrow\infty]{} 0$.
\\ Consequently, we get: 
 \[\E \left( \int_t^{+\infty} |g(s,Y_s,Z_s)| ds \right)^2\xrightarrow[t \rightarrow\infty]{} 0,\]
 where we have used the dominated convergence theorem (as $B_t^2:=\left( \int_t^{+\infty} |g(s,Y_s,Z_s)| ds \right)^2 \xrightarrow[t \rightarrow\infty]{} 0$ and $B_t^2 \leq B_0^2$, with $B_0^2 \in L^1$).
\\ Therefore, $ \E |Y_t - \eta|^2 \xrightarrow[t \rightarrow\infty]{} 0$, which is the desired property.
\end{Remark}

We can extend the above definition of the (non-linear) $g$-evaluation $\ce_S^g[\cdot]=\ce^g[\cdot|\cf_S]$ on an infinite horizon to the case where $S$ is a stopping time. 
\\
The $g$-conditional expectation of $\eta \in L^2(\mathcal{F}_{\infty})$ at time $S$, where $S$ is a stopping time (a.s. finite), is defined by
\begin{equation} \label{conditional_expectation_at_S}
   \ce^g[\eta|\mathcal{F}_{S}] := 
   \ce^g_S[\eta] := 
   Y_S,
\end{equation}
where $(Y_t)_{t \geq 0}$ is the first component of the solution of ${\rm BSDE}_{\infty} (\eta,g)$.
\\
The following properties hold true (for $\eta \in L^2(\mathcal{F}_{\infty})$):  
\begin{compactenum}
    \item[(i)] \emph{(usual zero-one law)}  %%Lemma 2.2 in \cite{chen2000infinite}
    $\; \ce_{S}^g[I_A\eta] =I_A\ce_{S}^g[\eta],$ for all $A\in\cf_S$.
    \item[(ii)] \emph{(monotonicity)} %Theorems A.2 and A.3 in \cite{chen2000infinite}
    If  $\eta_1\geq \eta_2$, then  $\ce_{S}^g[\eta_1]\geq \ce_{S}^g[\eta_2]$.
\item[(iii)] \emph{(strict monotonicity)} 
 If  $\eta_1\geq \eta_2$ and $P(\eta_1 > \eta_2)>0$, then $\ce_{S}^g[\eta_1]>\ce_{S}^g[\eta_2]$. 
\item[(iv)]  \emph{(knowledge preservation)} %Theorems A.4,(i) in \cite{chen2000infinite}
 If $\eta$ is moreover $\cf_S$-measurable, then  $\ce_{S}^g[\eta]=\eta$. 
\item[(v)] \emph{(time consistency)} For all $S, \tau $ such that $ S \leq \tau < +\infty $, $\;\ce_{S}^g[\ce_{\tau}^g[\eta]]= \ce_{S}^g[\eta]$. 
\item[(vi)] \emph{(continuity)}  %Theorems A.3, (4), in \cite{chen2000infinite}
For $(\eta_k), \eta \in L^2(\mathcal{F}_{\infty})$ such that $  || \eta_k - \eta||_{L^2(\mathcal{F}_{\infty})}  \xrightarrow[k \rightarrow +\infty]{} 0$, then \[ %\E[| \ce^g_{S}[\eta_k] - \ce^g_{S}[\eta]|^2]  \leq
|| \ce^g_{\cdot}[\eta_k] - \ce^g_{\cdot}[\eta]||_{S^2}  
\xrightarrow[k \rightarrow +\infty]{}  0.\]
\end{compactenum}
\noindent
\begin{Remark}
Let us prove that $\rho_S[\cdot]:=\ce^g[\cdot|\cf_S]$ satisfies the monotone Fatou property with respect to the terminal condition (property (vii) of $\rho$). 
Let $(\eta_{n}), \eta$ be such that $(\eta_{n})$ is non-decreasing, $\eta_{n} \in L^{2}(\cf_{\infty})$, $\sup_{n}\eta_{n} \in L^{2}(\cf_{\infty})$, and $\lim_{n \to +\infty} \uparrow \eta_{n} = \eta$ a.s.
Then, for each $n \in \N$, 
$\eta_0 \leq \eta_n \leq \sup_{m}\eta_{m}$.
Hence, for each $n \in \N$, 
\[|\eta_n| \leq \max(|\eta_0|,\sup_{m}\eta_{m}) \in L^2(\cf_{\infty}).\]
Hence, the sequence $(\eta_n)$ is bounded in $L^2(\cf_{\infty})$. 
Hence, by the monotone convergence, 
\[
 || \eta_n - \eta||_{L^2(\mathcal{F}_{\infty})}  \xrightarrow[n \rightarrow +\infty]{} 0.
\]
Hence, by property (vi) of the $g$-conditional expectation $\ce^g_S[\cdot]$, 
\[|| \ce^g_{\cdot}[\eta_n] - \ce^g_{\cdot}[\eta]||_{S^2}  \xrightarrow[n \rightarrow +\infty]{}  0.\]
Hence, 
\[ \E[| \ce^g_{S}[\eta_n] - \ce^g_{S}[\eta]|^2] \leq 
\E[ \sup_{t \geq 0}| \ce^g_{t}[\eta_n] - \ce^g_{t}[\eta]|^2]  \xrightarrow[n \rightarrow +\infty]{}  0.\] 
Hence, (up to a subsequence), $ \ce^g_{S}[\eta_n] \xrightarrow[n \rightarrow +\infty]{}  \ce^g_{S}[\eta]$ a.s.
\\ 
Moreover, $\ce^g_S[\eta_n]$ is non-decreasing, by the monotonicity property of $\ce^g_S[\cdot]$ (property (ii)).  
\\ The monotone Fatou property with respect to the terminal condition is thus proven. 
\\
Hence, the $g$-conditional expectation satisfies all the properties (i)-(vi) of the non-linear operators $\rho$.
\end{Remark}
\begin{Remark}\label{remark_C}
Let us prove that $\ce_{S}^g[\cdot]$ satisfies the following property: for $b>0$,
 \[
 \ce^g_{S}[\eta + b] 
 \leq  \ce^g_{S}[\eta]  + C b,
 \]
 where $C>0$ is a positive constant.\\
      By Lemma A.1 in \cite{chen2000infinite} (which uses a standard BSDE-linearization technique), there exists a probability measure $\mathbb{Q}$ and an adapted process $(a_t)$ which satisfies: for any $t\geq0$, $a_t \leq |a_t| \leq v(t)$, and such that: 
  \[
 \ce^g_{t}[\eta + b]-\ce^g_{t}[\eta]  = E_{\mathbb{Q}}[b e^{\int_t^{\infty} a_r dr}|\cf_t].
 \]
  Hence, as $b>0$,
   \[
 \ce^g_{t}[\eta + b]-\ce^g_{t}[\eta]  \leq E_{\mathbb{Q}}[b e^{\int_t^{\infty} v(r) dr}|\cf_t]
\leq b E_{\mathbb{Q}}[e^{\int_0^{\infty} v(r) dr}|\cf_t]
 \]
  As $v$ is a deterministic function, $E_{\mathbb{Q}}[ e^{\int_0^{\infty} v(r) dr}|\cf_t]= e^{\int_0^{\infty} v(r) dr}$.
  Hence,  for each $t \geq 0$,
    \[
 \ce^g_{t}[\eta + b]-\ce^g_{t}[\eta]  
\leq b e^{\int_0^{\infty} v(r) dr}.
 \]
  We set $C:= e^{\int_0^{\infty} v(r) dr}$. 
  We have: $0< C <+\infty$ (as, by (H2), $v$ is integrable).
  We note that $C$ depends on $\rho$ only (via the function $v$).
  Hence, for each $t \geq 0$, $  
 \ce^g_{t}[\eta + b]
 \leq  \ce^g_{t}[\eta] + Cb$. 
 Hence, for each a.s. finite stopping time $S$,  $  
 \ce^g_{S}[\eta + b]
 \leq  \ce^g_{S}[\eta] + Cb$ a.s.
\end{Remark}
By the above, we can apply our result in the case where $\rho_{S,\tau}:=\rho_{S}=\ce^g_S$.
We need however to ensure that the value family is in $L^2$. 
\begin{Assumption}\label{Hyp3}
For each   $\nu \in \Theta$, the random variable $V(\nu)$ is in $L^2$.
\end{Assumption}
 The following remark can be found in \cite{Grigorova-5} and remains valid in our framework as well, applied with $\rho_{S,\tau}:=\rho_{S}=\ce^g_S$.
\begin{Remark} \label{TEXRemaRk24}
Let $\rho$ satisfy the assumptions of admissibility (ii), knowledge preservation (iii), "generalized" zero-one law (vi), and monotonicity (iv). If the pay off family $\xi = (\xi(\tau))_{\tau \in \Theta}$ is $L^2$- integrable and dominated from above by an $L^2$-integrable $(\Theta, \rho)$-martingale $M$, then the value family $V$ satisfies the integrability Assumption \ref{Hyp3}.\\
Indeed, let $S \in \Theta$ be given. By Remark \ref{TEXRemark33}, Statement 2, $V(S) \geq \xi(S)$.\\ 
On the other hand, by assumption on $\xi$, for each $\tau \in \Theta_{S}$, $\xi(\tau) \leq M(\tau)$. Hence, by monotonicity of $\rho$, we have $\rho_{S, \tau}[\xi(\tau)] \leq \rho_{S, \tau}[M(\tau)] = M(S),$
where we have used the $(\Theta, \rho)$-martingale property of $M$ for the last equality.\\
So, $V(S) = \esssup_{\tau \in \Theta_{S}}\rho_{S, \tau}[\xi(\tau)] \leq M(S)$. Hence, we get $\xi(S) \leq V(S) \leq M(S)$, which proves that $V(S) \in L^{2}$. Therefore, Assumption \ref{Hyp3} is satisfied.
\\ In the case where $\rho_{S,\tau}:=\rho_{S}=\ce^g_S$, this condition amounts to the pay-off family $\xi = (\xi(\tau))_{\tau \in \Theta}$ being dominated by an $L^2$-integrable $(\Theta, \ce^g)$-martingale. 
\end{Remark}

\section{Appendix A}\label{Appendix}
\begin{Remark} \label{Remark_closed_under_limit}
\begin{itemize}
\item[a)] The case of a non-increasing sequence of stopping times $(\tau_n)$ in $ \Theta^{\mathbb{N}}$. \\
For each non-increasing sequence of stopping times 
$(\tau_n)_{n \in {\mathbb N}} \in \Theta^{\mathbb{N}}$, we have $\lim_{n \rightarrow + \infty} \tau_n$  $\in$ $\Theta$.
  Indeed, let $\omega \in \Omega$, for each $n$, there exist $l_n=l_n(\omega)$ such that $\tau_n(\omega)=\theta_{l_n}(\omega)$.
  Moreover, for each $n \in \N$, $0 \leq \tau_n(\omega) \leq \tau_1(\omega) < +\infty $.
  Hence, for each $n \in \N$, $\theta_{l_n}(\omega) \leq \tau_1(\omega) < +\infty $.
  \\ As $\theta_k(\omega) \uparrow +\infty$, there are only a finite number (depending on $\omega$) of time points $\theta_k(\omega)$ such that $\theta_k(\omega) \leq \tau_1(\omega)=\theta_{l_1}(\omega)$. 
  Hence, all elements of the sequence $(\tau_n(\omega))_{n \in \N}$ are valued in $\{0, \theta_1(\omega), \theta_2(\omega), \dots, \theta_{l_1}(\omega)\}$. 
   Therefore, we have $\lim_{n \rightarrow + \infty} \tau_n (\omega) \in \{0, \theta_1(\omega), \theta_2(\omega), \dots, \theta_{l_1}(\omega)\}$, which implies that from a certain rank onward, the sequence is constant. 
   The desired property is thus proved.

\item[b)] The case of a non-decreasing sequence of stopping times $(\tau_n)$ in $ \Theta^{\mathbb{N}}$. 
\begin{itemize}
    \item 
Let $(\tau_n) \subset \Theta^{\mathbb{N}}$ be a non-decreasing sequence of stopping times in $ \Theta^{\mathbb{N}}$.
Let $\tau$ be the (a.s.) limit of $(\tau_n)$. 
Let us fix $\omega \in \Omega$. 
For each $n \in \N$, there exists $k_n=k_n(\omega) \in \N$ such that $\tau_n(\omega)=\theta_{k_n}(\omega)$.
As $\tau_n (\omega)\leq \tau_{n+1}(\omega) $, we have $\theta_{k_n}(\omega) \leq \theta_{k_{n+1}}(\omega) $. 
Moreover, if for all $n \in \N$, $\tau_n (\omega)< \tau_{n+1} (\omega)$, then $\theta_{k_n}(\omega) < \theta_{k_{n+1}}(\omega) $. 
As the sequence $(\theta_k (\omega))$ is non-decreasing, we get $k_n < k_{n+1} $. 
This reasoning shows that $\lim_{n \rightarrow + \infty} \tau_n$ might converge to $+\infty$.

\item  If $(\tau_n)$ is a non-decreasing sequence of stopping times such that $\tau_n \uparrow \tau$ a.s. and 
$\tau$ is a.s. finite, then $\tau \in \Theta$.
\\ Indeed, we have $\tau$ is a stopping time (a.s. finite). 
Let us fix $\omega \in \Omega$. 
 For each $n \in \N$, $\tau_n(\omega)= \theta_{k_n}(\omega) \leq \tau(\omega)$. 
Since $\theta_k (\omega) \uparrow +\infty$, there are a finite number (depending on $\omega$) of time points $\theta_k(\omega)$ satisfying $\theta_{k}(\omega) \leq \tau(\omega)$. 
Then, as $(\tau_n(\omega))_{n \in \mathbb{N}}$ is non-decreasing and taking values in a finite set, from a certain rank onward, the sequence is constant. That is, after a certain rank $n_0 =n_0(\omega)$, we have: $\tau_n(\omega)=\theta_{k_{n_0}}(\omega)$ for all $n \geq n_0$. 
Consequently, $\tau(\omega)=\limn \tau_n (\omega)=\theta_{k_{n_0}}(\omega)$.
\end{itemize}
\end{itemize}
  
\end{Remark}

\begin{proof}[Proof of Theorem \ref{Theorem_supermartingale_Snellenvelope}]
    By Lemma \ref{Lemma_admissible},  the value family $V$ is admissible.
    \\
    Let $S \in \Theta_{S}$ and $\tau \in \Theta_{S}$.
    To show the \emph{$(\Theta, \rho)$}-supermartingale property of the value family,
    it remains to show $\rho_{S,\tau}[V({\tau})]\leq V({S})$ a.s.
    By the maximizing sequence lemma (cf. Lemma \ref{Maximizing_sequence_lemma}), 
      there exists a sequence $(\tau_p) \in (\Theta_{\tau})^{\mathbb{N}}$,
      such that $V(\tau)= \lim_{p \rightarrow + \infty}\uparrow \rho_{\tau, \tau_p}[\xi(\tau_p)]$.
      Hence, we have: 
      \[
      \rho_{S, \tau}[V(\tau)]= \rho_{S, \tau} [\lim_{p \rightarrow + \infty}\uparrow \rho_{\tau, \tau_p}[\xi(\tau_p)]] 
      \leq  \liminf_{p \rightarrow + \infty} \rho_{S, \tau} [\rho_{\tau, \tau_p}[\xi(\tau_p)],
      \]
      where we have used the monotone Fatou property with respect to the terminal condition (vii) in the last inequality.
      By the consistency property, we have: 
        \[
       \liminf_{p \rightarrow + \infty} \rho_{S, \tau} [\rho_{\tau, \tau_p}[\xi(\tau_p)]
       =  \liminf_{p \rightarrow + \infty} \rho_{S,\tau_p}[\xi(\tau_p)]
       \leq V(S),
      \]
      where the last equality is due to $\Theta_{\tau} \subset \Theta_S$.
      We conclude that $\rho_{S,\tau}[V({\tau})]\leq V({S})$.
      Hence, the value family $V$ is a \emph{$(\Theta, \rho)$}-supermartingale family.
      This proves Statement 1 of the theorem.
      \\
      Let us now show Statement 2.
      By Remark \ref{TEXRemark33}, Statement 2, we have $V \geq \xi$.
      By Statement 1, we have that $V$ is a \emph{$(\Theta, \rho)$}-supermartingale.
       It remains to show that $V$ is the smallest. 
      Let $(V'(\tau))$ be another \emph{$(\Theta, \rho)$}-supermartingale family, such that, for each $\tau \in \Theta$,
      $V'(\tau) \geq \xi(\tau)$ a.s.
      Let $S \in \Theta$, $\tau \in \Theta_S$.
      By the monotonicity of the non-linear operators, we have:
      \[
      \rho_{S,\tau} [V'(\tau)] \geq  \rho_{S,\tau}[\xi(\tau)].
      \]
      On the other hand, as $(V'(\tau))$ is a \emph{$(\Theta, \rho)$}-supermartingale family, 
      we have $V'(S) \geq  \rho_{S,\tau} [V'(\tau)]$.
      Hence, 
      $V'(S) \geq  \rho_{S,\tau} [V'(\tau)] \geq  \rho_{S,\tau}[\xi(\tau)]$.
      By taking the essential supremum over $\tau \in \Theta_S$, we get: 
      \[
      V'(S) \geq  \esssup_{\tau \in \Theta_S}  \rho_{S,\tau}[\xi(\tau)]
      = V(S)
      \]
      The proof is complete.
\end{proof}

In the following lemma, we give a useful technical result that has been established in \cite{Grigorova-5} and hold true also in our infinite horizon framework.

\begin{lemma}\label{Lemma_phi_martingale_on_stochastic_interval}
    Let $\rho$ satisfy the consistency property (v).
    Let $\phi=(\phi(\nu))$ be a given non-negative admissible family.
    Let $S$ $\in$ $\Theta$ and $\tau$ $\in$ $\Theta$ be such that $S \leq \tau$ a.s. 
    We assume that for any  $\sigma$ $\in$ $\Theta$ such that $S \leq \sigma \leq \tau$ a.s., it holds 
    \begin{equation}\label{eq_5}
        \rho_{\sigma,\tau}[\phi(\tau)] = \phi(\sigma) \,\,\,\, a.s.
    \end{equation}
    Then, $\phi$ is a \emph{$(\Theta, \rho)$}-martingale on the stochastic interval $[S,\tau]$, that is, for any  $\nu_1$ $\in$ $\Theta$, and $\nu_2$ $\in$ $\Theta$, such that $S \leq \nu_1 \leq \nu_2 \leq \tau$ a.s.,
    \begin{equation}
          \rho_{\nu_1,\nu_2}[\phi(\nu_2)] = \phi(\nu_1) \,\,\,\, a.s.
    \end{equation}
\end{lemma}
\begin{proof}
    Let $\nu_1$ $\in$ $\Theta$, and $\nu_2$ $\in$ $\Theta$, such that $S \leq \nu_1 \leq \nu_2 \leq \tau$ a.s.
     By applying Eq. \eqref{eq_5} with $\sigma = \nu_1$ and by the consistency of the non-linear operators $\rho$, we have: 
    \[
     \phi(\nu_1) =  \rho_{\nu_1,\tau}[\phi(\tau)] 
     = \rho_{\nu_1,\nu_2}[ \rho_{\nu_2,\tau}[\phi(\tau)] ] \,\,\,\, a.s.
    \]
    Then, by applying again Eq. \eqref{eq_5} with $\sigma = \nu_2$, we have
    \[
     \rho_{\nu_2,\tau}[\phi(\tau)]  =  \phi(\nu_2).
    \]
    Hence, $   \phi(\nu_1)
     = \rho_{\nu_1,\nu_2}[ \rho_{\nu_2,\tau}[\phi(\tau)] ] 
     = \rho_{\nu_1,\nu_2}[\phi(\nu_2)] $ a.s.
\end{proof}

\begin{proof}[Proof of Proposition \ref{prop_optimality_criterion}]  Let $\nu^*$ $\in$ $\Tk$ be such that  
$V(\nu^*)=\xi(\nu^*)$ a.s., and 
the family $(V(\nu \wedge \nu^*))_{\nu \in \Tk}$ is a \emph{$(\Theta, \rho)$}-martingale family. \\ Hence, for any $\nu$ $\in$ $\Tk$, we have: \[    V(\theta_k \wedge \nu^*) = \rho_{\theta_k \wedge \nu^*,\nu \wedge \nu^*}[V(\nu \wedge \nu^*)] \,\,\,\, a.s.,  \]    which implies  (as $\theta_k \wedge \nu^* = \theta_k$, which follows from the fact that $\nu^* \geq \theta_k$)    \[V(\theta_k) = \rho_{\theta_k,\nu \wedge \nu^*}[V(\nu \wedge \nu^*)] \,\,\,\, a.s.\]    In particular, for $\nu=\nu^*$, we get  \[  V(\theta_k) = \rho_{\theta_k,\nu^* \wedge \nu^*}[V(\nu^* \wedge \nu^*)]    = \rho_{\theta_k,\nu^*}[V(\nu^*)]\,\,\,\, a.s.\]    From this, together with condition ii), we have  \[ V(\theta_k)  = \rho_{\theta_k,\nu^*}[V(\nu^*)]  = \rho_{\theta_k,\nu^*}[\xi(\nu^*)]\,\,\,\, a.s.,  \]which implies that the stopping time $\nu^*$ is an optimal stopping time for problem \eqref{optimal_stopping_problem_kbis}.   The proof is complete. \end{proof}


\begin{thebibliography}{99}

\bibitem[Bayraktar et al.(2010)]{Bayraktar-2}  E. Bayraktar,  I. Karatzas,  and S. Yao,  \textit{Optimal stopping for dynamic convex risk measures}, Illinois Journal of Mathematics  54 (2010), pp. 1025-1067. 

\bibitem[Bayraktar and Yao-Part I(2011)]{Bayraktar} E. Bayraktar and S. Yao, \textit{Optimal stopping for non-linear expectations Part I},  Stochastic Processes and Their Applications 121  (2011), pp. 185-211.

\bibitem[Bayraktar and Yao-Part II(2011)]{Bayraktar-3} E. Bayraktar and S. Yao, \textit{Optimal stopping for non-linear expectations Part II}, Stochastic Processes and Their Applications 121  (2011), pp. 212-264.


\bibitem[Belomestny and Kr{\"a}tschmer(2016)]{belomestny2016optimal} D. Belomestny and V. Kr{\"a}tschmer,
\textit{Optimal stopping under model uncertainty: Randomized stopping times approach}, Ann. Appl. Probab. volume 26(2), (2016), pp.1260-1295.

\bibitem[Belomestny and Kr{\"a}tschmer(2017)]{belomestny2017addendum} D. Belomestny and V. Kr{\"a}tschmer,
\textit{Addendum to "Optimal stopping under model uncertainty: Randomized stopping times approach"}, Ann. Appl. Probab. volume 27(2), (2017), pp.1289-1293.

%% \bibitem[Bion-Nadal(2009)]{Biagini} J. Bion-Nadal, \textit{Time consistent dynamic risk processes}, Stochastic Processes and Their Applications 119 (2009), pp. 633-654.

%%\bibitem[Cheridito et al.(2006)]{Cheridito} P. Cheridito, F. Delbaen, and M. Kupper, \textit{Dynamic monetary risk measures for bounded discrete-time processes}, Electron. J. Probab. 11 (2006), pp. 57-106 (electronic).


\bibitem[Chen and Wang (2000)]{chen2000infinite} Z. Chen and B. Wang, \textit{Infinite time interval BSDEs and the convergence of g-martingales}, Journal of the Australian Mathematical Society, volume 69(2), (2000), pp.187-211. 

%% \bibitem[Delbaen et al.(2010)]{Delbaen} F. Delbaen, S. Peng and E. Rosazza - Gianin, \textit{Representation of the penalty term of dynamic concave utilities}, Finance Stoch. 14 (2010), pp. 449-472. %MR 2670421

%% \bibitem[Dumitrescu et al.(2016)]{Dumitrescu-Quenezz-Sulem} R. Dumitrescu, M.C. Quenez and A. Sulem, \textit{Mixed generalized Dynkin game and stochastic control in a Markovian framework}, Stochastics (2016), pp. 400-429. %(https://doi.org/10.1080/17442508.2016.1230614).

%% \bibitem[El Karoui (1981)]{ElKaroui_notes} N. El Karoui, \textit{Les aspects probabilistes du contr$\hat{o}$le stochastique}, in: Ecole d'\' et\' e de Saint-Flour,  Lecture Notes in Math., vol. 876, (1981), Springer, Berlin, , pp. 73-238.   
%% \bibitem[El Karoui et al.(1997)]{ElKaroui}  N. El Karoui, S. Peng, and M. C. Quenez, \textit{Backward stochastic differential equations in finance}, Mathematical Finance 7  (1997), pp. 1-71.

\bibitem[El Karoui and Quenez (1997)]{EQ}
N. El Karoui and M.-C. Quenez, %(1996):
\textit{Non-linear pricing theory and backward stochastic differential 
equations},  In Financial Mathematics, 
Lectures Notes in Mathematics 1656, %Bressanone, 1996, 
Ed. W. Runggaldier,  Springer, 1997.

\bibitem[Ekren et al.(2014)]{Ekren} I. Ekren, N. Touzi and J. Zhang,
\textit{Optimal stopping under nonlinear expectation},
Stochastic Processes and their Applications, volume 124(10),  (2014), pp. 3277-3311.

\bibitem[Grigorova et al.(2017)] {Grigorova-1} M. Grigorova,  P. Imkeller , E. Offen, Y. Ouknine, and M.-C. Quenez, \textit{Reflected BSDEs when the obstacle is not right-continuous and optimal stopping}, Annals of Applied Probability, volume 25(5), (2017), pp. 3153-3188. 

\bibitem[Grigorova et al.(2020)] {Grigorova-3} M. Grigorova, P. Imkeller, Y. Ouknine and M.-C. Quenez, \textit{Optimal stopping with f-expectations: The irregular case}, Stochastic Processes and their Applications, volume 130 (3), (2020), pp. 1258-1288.

\bibitem[Grigorova et al.(2020)2] {Grigorova-4} M. Grigorova, P. Imkeller, Y. Ouknine and M.-C. Quenez, \textit{On the strict value of the non-linear optimal stopping problem}, Electronic Communications in Probability, volume 25, (2020), paper 49, 9 pages.

\bibitem[Grigorova and Quenez(2016)]{Grigorova-2} M. Grigorova and M.-C. Quenez, \textit{Optimal stopping and a non-zero-sum Dynkin game in discrete time with risk measures induced by BSDEs},  Stochastics (2016), pp. 259-279. %Informa, UK, Limited. $(http://dx.doi.org/10.1080/17442508.2016.1166505)$. 
\bibitem[Grigorova et al.(2021)]{Grigorova_Am} M. Grigorova, M.-C. Quenez and A. Sulem,  \textit{American options in a non-linear incomplete market model with default}, Stochastic Processes and their Applications, volume 142, (2021), pp. 479-512. 

\bibitem[Grigorova et al.(2024)] {Grigorova-5} M. Grigorova, M.-C. Quenez and P. Yuan, \textit{Optimal stopping: Bermudan strategies meet non-linear evaluations}, Electronic Journal of  Probability, volume 29, (2024), pp. 1 - 29.

%% \bibitem[Grigorova et al.(2025)b] {Grigorova-6} M. Grigorova, M.-C. Quenez and P. Yuan, \textit{The non-linear multiple stopping problem: between the discrete and the continuous time}, arXiv preprint arXiv:2504.13503 (2025).

\bibitem[Grigorova et al.(2025)a] {Grigorova-7} M. Grigorova, M.-C. Quenez and P. Yuan, \textit{Non-linear non-zero-sum Dynkin games with Bermudan strategies}, Journal of Optimization Theory and Applications, volume 206(1), (2025), pp. 1 - 20.

\bibitem[Kim et al.(2021)]{Rutkowski} E. Kim, T. Nie and M. Rutkowski, \textit{American options in nonlinear markets}, Electronic Journal of  Probability, volume 26, (2021), pp. 1-41.
\bibitem[Klimsiak and Rzymowski(2021)]{Klimsiak} T. Klimsiak and M. Rzymowski, \textit{Reflected BSDEs with two optional barriers and
monotone coefficient on general filtered space}
Electronic  Journal of Probability 26, article no. 91, (2021), pp. 1-24.
\bibitem[Kr\"atschmer and Schoenmakers(2010)]{Schoe}  V. Kr\"atschmer and J. Schoenmakers, \textit{Representations for optimal stopping under dynamic monetary utility functionals}, SIAM Journal on Financial Mathematics 1 (2010), pp. 811-832.

\bibitem[Menaldi and Robin(2016)]{menaldi2016Robin} J.L. Menaldi and M. Robin, \textit{On some optimal stopping problems with constraint}, SIAM Journal on Control and Optimization, volume 54(5), (2016), pp. 2650-2671. %doi: 10.1137/15M1040001

\bibitem[Neveu(1975)]{Neveu}  J. Neveu,  \textit{Discrete-parameter martingales}, North-Holland, Amsterdam, 1975. 


\bibitem[Nutz and Zhang (2015)]{Nutz} M. Nutz and J.  Zhang, \textit{Optimal Stopping under adverse non-linear expectations and related games},  The Annals of Applied Probability, volume 25(5),  (2015), pp. 2503-2534. %http://www.jstor.org/stable/24521609


%%\bibitem[Pardoux and Peng(1990)]{Pape90}  E. Pardoux and  S. Peng,  \textit{Adapted solution of backward stochastic differential equation},Systems \& Control Letters 14 (1990), pp. 55-61.

%%\bibitem[Peng(1997)]{Peng-2} S. Peng, \textit{Backward SDE and related g-expectation}, in \textit{in Backward Stochastic Differential Equations}, Pitman Research Notes in Math. Series, No. 364, El Karoui Mazliak edit, (1997) pp. 141-159.

%%\bibitem[Peng(2004)]{Pe04}  S.  Peng, \textit{Nonlinear expectations, nonlinear evaluations and risk measures}, Lecture Notes in Math., 1856, Springer, Berlin,  (2004), pp. 165-253.
 
%%\bibitem[Peng(2004)]{Peng-1} S. Peng, \textit{Backward stochastic differential equations, nonlinear expectations, nonlinear evaluations, and risk measures}, Lecture notes in Chinese Summer School in Mathematics, Weihei, (2004).  

%%\bibitem[Quenez and Kobylanski(2012)]{Quenez-Kobylanski} M. C. Quenez and M. Kobylanski, \textit{Optimal stopping time problem in a general framework}, Electron. J. Probab. 17, no. 72 (2012), pp. 1-28.

%%\bibitem[Quenez and Sulem(2013)]{Quenez-Sulem-0}  M. C. Quenez and A. Sulem, \textit{BSDEs with jumps, optimization and applications to dynamic risk measures}, Stochastic Processes and their Applications 123  (2013), pp. 3328-3357.

\bibitem[Quenez and Sulem(2014)]{Quenez-Sulem}  M. C. Quenez and A. Sulem, \textit{Reflected BSDEs and robust optimal stopping for dynamic risk measures with jumps}, Stochastic Processes and their Applications 124 (2014), pp. 3031-3054.



\end{thebibliography}
\end{document}